\documentclass[10pt]{article}
\usepackage{e-jc}
\usepackage{pgf,tikz}
\usetikzlibrary{arrows}
\usepackage{amsmath}%
\usepackage{amsfonts}%
\usepackage{amssymb}%
\usepackage{graphicx}

\usepackage{amsthm,amsmath,amssymb}

\usepackage{graphicx}

\usepackage[colorlinks=true,citecolor=black,linkcolor=black,urlcolor=blue]{hyperref}

\theoremstyle{plain}
\newtheorem{theorem}{Theorem}
\newtheorem{lemma}[theorem]{Lemma}

\theoremstyle{definition}

\theoremstyle{remark}

\title{Diagonal Ramsey numbers in multipartite graphs related to stars}
\author{Chula Jayawardene\\
\small  Department of Mathematics\\[-0.8ex]
\small  University of Colombo, Sri Lanka.\\
\small\tt c\_jayawardene@yahoo.com\\
}
\date{}
\begin{document}

\maketitle
\begin{abstract}
Abstract: Let the star on $n$ vertices, namely $K_{1,n-1}$ be denoted by $S_n$. If every two coloring of the edges of a complete balanced multipartite graph $K_{j \times s}$ there is a copy of $S_n$ in the first color or a copy of $S_m$ in the second color, then we will say $K_{j \times s}\rightarrow (S_n,S_m)$. The size Ramsey multipartite number $m_j(S_n, S_m)$ is the smallest natural number $s$ such that $K_{j \times s}\rightarrow (S_n,S_m)$. In this paper, we obtain the \textbf{exact} values of the size Ramsey numbers $m_j(S_n,S_m)$ for $n,m \ge 3$ and $j \ge 3$.
\end{abstract}

\section{\textbf{Introduction}}

In this paper we concentrate on simple graphs without loops and multiple edges. Let the complete multipartite graph having  $j$ uniform sets of size $s$ be denoted by $K_{j \times s}$ and the complete bipartite graph on $n+m$ vertices be denoted by $K_{n,m}$. Given, three graphs $K_N$, $G$ and $H$, we say that $K_N \rightarrow (G, H)$ if $K_N$ is colored by two colors red and blue and it contains a copy of G(in the first color) or a copy of H(in the second color). Using this notation we define the classical Ramsey number $r(n, m)$ as the smallest integer $N$ such that $K_N \rightarrow (K_n, K_m)$. Sadly to say, even in the case of diagonal classical Ramsey numbers $r(n,n)$ almost nothing significant is known beyond the case $n=5$ (see  Radziszowski et al 2014  for a survey). 

\vspace{8pt}

\noindent In the decades that followed there are several interesting variations that have originated from these classical Ramsey number. One obvious variation is the case of size Ramsey numbers  build up mainly by Erd{\"o}s, Faudree, Rousseau and Shelph (see Erd{\"o}s et al 1978  and  Faudree et al 1975). Another variation introduced recently by  Buger et al 2004 and Syafrizal et al  2005, is the concept of balanced multipartite Ramsey numbers. This concept is based on exploring the two colorings of multipartite graphs $K_{j \times s}$ instead of the complete graph. Formally define size Ramsey multipartite number $m_j(G, H)$ as the smallest natural number $s$ such that $K_{j \times s}\rightarrow (G,H)$. However, sad to say currently there are very few known balanced multipartite Ramsey numbers between pairs of  graphs and pairs of classes of graphs other than the ones introduced initially by Syafrizal et al  2005 and 2009). 

\vspace{8pt}

\maketitle

\section*{\textbf{Notation}}

Given a graph $G=(V,E)$ with the \textit{order} of the graph  is denoted by $|V(G)|$ and  the \textit{size} of the graph is denoted by $|E(G)|$. For a vertex $v$ of a graph $G$, the \textit{neighborhood} of $v$ is denoted by $N(v)$ and is defined as the set of vertices adjacent to $v$. Further the cardinality of this set, denoted $d(v)$,  is defined as the \textit{degree} of $v$. We say that a graph $G$ is a  $k$-regular graph if $d(v)= k$ for all $v\in V(G)$. Given a red-blue coloring of $K_{j \times s} =H_R \oplus H_B$.  The red degree and blue degree of any vertex $v$ belonging to  $V(K_{j \times s})=\{$$v_{k,i} $$|$ $ 0 \le i \le s-1,$ $ 0 \le k \le j-1\}$ denoted by  $d_R(v)$ and $d_B(v)$ respectively, are defined as the degree of vertex $v$ in $H_R$ and $H_B$ respectively.

\vspace{6pt}

\noindent  Given $w \ge2$, $0\le i \le w-1$ and $0 <c \le w-1$, define
$\sigma_{c,w}(i) = \{a_1\} \cup \{a_2\}$, $\sigma^{+}_{c,w}(i) = \{a_1\}$,  $\sigma^{-}_{c,w}(i) = \{a_2\}$ and $B_{0,w}(i)=\phi$ and if $k>0$, $B_{k,w}(i)=\cup_{c=1}^{k} \sigma_{c,w}(i)$ where 
\[
a_1 = 
\begin{cases} 
i+c & \text{ if }  i+c \le w-1   \\
\hspace{20pt} & \\
i+c-w &  \text{ if }  i+c>w-1 \\
\end{cases}
\]
\vspace{8pt}
\[
a_2 = 
\begin{cases} 
i-c & \text{ if }  i-c \ge 0   \\
\hspace{20pt} & \\
i-c+w &  \text{ if }  i-c<0 \\
\end{cases}
\]

\vspace{2pt}

\section{\textbf{Some Lemmas } }
\noindent In all the following lemmas assume $d > 0$ as the results are trivially true when $d=0$.
\begin{lemma}
\label{lemma 1}
There exists a regular induced subgraph of degree $d$ of $K_{j \times s}$ on the vertex set $V(K_{j \times s})$ provided  that $d$ is even, $j$ is odd and $s$ is odd.
\end{lemma}

\begin{proof}

\noindent If $d=2k_1(j-1)+2k_2$ for some non negative integers $k_1$ and $k_2$ such that $2k_1 \le s-1$ and $0<2k_2 \le j-1$. Construct  a graph by connecting the vertices $v_{i,l}$ and $v_{p,r}$ if one of the following situations hold

\vspace{3pt}

\noindent \textbf{a)} If $r \in B_{k_1,s}(l)$ and $p \neq i$.
\hspace{14pt} \textbf{b)} If $r=l$ and $p \in B_{k_2,j}(i)$.

\vspace{8pt}

\begin{tikzpicture}[line cap=round,line join=round,>=triangle 45,x=1.0cm,y=1.0cm]
\clip(-4.139000256081949,0.22752219679634095) rectangle (5.842508066581311,5.571668649885585);
\draw (0.5962545454545466,0.7514581235697962) node[anchor=north west] {$v_{2,2}$};
\draw [rotate around={31.319407985551642:(-1.8820892445582504,4.232721281464539)}] (-1.8820892445582504,4.232721281464539) ellipse (1.8193170789947408cm and 0.5790787938258259cm);
\draw [rotate around={-39.71631220657896:(3.7474133162612024,3.9940393592677355)}] (3.7474133162612024,3.9940393592677355) ellipse (1.7494760444172837cm and 0.7202335067485682cm);
\draw (-1.4477403329065306,0.832959267734556) node[anchor=north west] {$v_{2,4}$};
\draw (-0.42574289372599194,0.7863871853546933) node[anchor=north west] {$v_{2,3}$};
\draw (4.6842443021767,4.547082837528606) node[anchor=north west] {$v_{1,2}$};
\draw (-3.423602048655572,4.803229290617851) node[anchor=north west] {$v_{0,2}$};
\draw (1.5160522407170314,0.7514581235697962) node[anchor=north west] {$v_{2,1}$};
\draw (5.19524302176697,3.6738562929061804) node[anchor=north west] {$v_{1,4}$};
\draw (5.007876824583871,4.1046480549199105) node[anchor=north west] {$v_{1,3}$};
\draw (4.292478617157495,4.931302517162473) node[anchor=north west] {$v_{1,1}$};
\draw (-3.849434314980796,4.372437528604121) node[anchor=north west] {$v_{0,1}$};
\draw (2.2314504481434083,0.7863871853546933) node[anchor=north west] {$v_{2,0}$};
\draw (-2.4016046094750334,5.431952402745997) node[anchor=north west] {$v_{0,4}$};
\draw (-2.8785367477592847,5.105947826086958) node[anchor=north west] {$v_{0,3}$};
\draw (3.7985131882202343,5.2922361556064095) node[anchor=north west] {$v_{1,0}$};
\draw (-4.087900384122922,3.941645766590391) node[anchor=north west] {$v_{ 0,0}$};
\draw [rotate around={-0.24175315415687906:(0.6984542893726006,1.3801812356979415)}] (0.6984542893726006,1.3801812356979415) ellipse (2.828566326316242cm and 0.621612178572758cm);
\draw (-0.9367416133162614,1.3336091533180803)-- (-3.236235851472471,3.673856292906182);
\draw (-0.9367416133162614,1.3336091533180803)-- (-1.9417057618437896,4.605297940503436);
\draw (-0.9367416133162614,1.3336091533180803)-- (3.1853147247119082,4.954588558352407);
\draw (-0.9367416133162614,1.3336091533180803)-- (4.496878104993598,3.860144622425633);
\draw (-0.0850770806658131,1.3219661327231145)-- (-2.3845713188220223,4.3375084668192265);
\draw (-0.0850770806658131,1.3219661327231145)-- (-1.4307070422535209,4.942945537757441);
\draw (-0.0850770806658131,1.3219661327231145)-- (4.054012548015366,4.232721281464535);
\draw (-0.0850770806658131,1.3219661327231145)-- (4.956776952624839,3.406066819221971);
\draw (0.7495541613316263,1.2986800915331826)-- (-2.8104035851472466,3.999860869565221);
\draw (0.7495541613316263,1.2986800915331826)-- (-1.9417057618437896,4.605297940503436);
\draw (0.7495541613316263,1.2986800915331826)-- (3.5941137003841233,4.593654919908472);
\draw (0.7495541613316263,1.2986800915331826)-- (4.496878104993598,3.860144622425633);
\draw (1.6523185659411002,1.263751029748286)-- (-3.236235851472471,3.673856292906182);
\draw (1.6523185659411002,1.263751029748286)-- (-2.3845713188220223,4.3375084668192265);
\draw (1.6523185659411002,1.263751029748286)-- (3.1853147247119082,4.954588558352407);
\draw (1.6523185659411002,1.263751029748286)-- (4.054012548015366,4.232721281464535);
\draw (2.333650192061459,1.2753940503432515)-- (-2.8104035851472466,3.999860869565221);
\draw (2.333650192061459,1.2753940503432515)-- (-1.4307070422535209,4.942945537757441);
\draw (2.333650192061459,1.2753940503432515)-- (3.5941137003841233,4.593654919908472);
\draw (2.333650192061459,1.2753940503432515)-- (4.956776952624839,3.406066819221971);
\draw (-1.4307070422535209,4.942945537757441)-- (4.496878104993598,3.860144622425633);
\draw (-1.4307070422535209,4.942945537757441)-- (3.1853147247119082,4.954588558352407);
\draw (-1.9417057618437896,4.605297940503436)-- (4.054012548015366,4.232721281464535);
\draw (-1.9417057618437896,4.605297940503436)-- (4.956776952624839,3.406066819221971);
\draw (-2.3845713188220223,4.3375084668192265)-- (3.5941137003841233,4.593654919908472);
\draw (-2.3845713188220223,4.3375084668192265)-- (4.496878104993598,3.860144622425633);
\draw (-2.8104035851472466,3.999860869565221)-- (3.1853147247119082,4.954588558352407);
\draw (-2.8104035851472466,3.999860869565221)-- (4.054012548015366,4.232721281464535);
\draw (-3.236235851472471,3.673856292906182)-- (3.5941137003841233,4.593654919908472);
\draw (-3.236235851472471,3.673856292906182)-- (4.956776952624839,3.406066819221971);
\draw (-3.236235851472471,3.673856292906182)-- (3.1853147247119082,4.954588558352407);
\draw (-3.236235851472471,3.673856292906182)-- (2.333650192061459,1.2753940503432515);
\draw (-2.8104035851472466,3.999860869565221)-- (3.5941137003841233,4.593654919908472);
\draw (-2.8104035851472466,3.999860869565221)-- (1.6523185659411002,1.263751029748286);
\draw (-2.3845713188220223,4.3375084668192265)-- (4.054012548015366,4.232721281464535);
\draw (-2.3845713188220223,4.3375084668192265)-- (0.7495541613316263,1.2986800915331826);
\draw (-1.9417057618437896,4.605297940503436)-- (4.496878104993598,3.860144622425633);
\draw (-1.9417057618437896,4.605297940503436)-- (-0.0850770806658131,1.3219661327231145);
\draw (-1.4307070422535209,4.942945537757441)-- (4.956776952624839,3.406066819221971);
\draw (-1.4307070422535209,4.942945537757441)-- (-0.9367416133162614,1.3336091533180803);
\draw (3.1853147247119082,4.954588558352407)-- (2.333650192061459,1.2753940503432515);
\draw (3.5941137003841233,4.593654919908472)-- (1.6523185659411002,1.263751029748286);
\draw (4.054012548015366,4.232721281464535)-- (0.7495541613316263,1.2986800915331826);
\draw (4.496878104993598,3.860144622425633)-- (-0.0850770806658131,1.3219661327231145);
\draw (4.956776952624839,3.406066819221971)-- (-0.9367416133162614,1.3336091533180803);
\begin{scriptsize}
\draw [fill=black] (-1.4307070422535209,4.942945537757441) circle (1.5pt);
\draw [fill=black] (-1.9417057618437896,4.605297940503436) circle (1.5pt);
\draw [fill=black] (-2.8104035851472466,3.999860869565221) circle (1.5pt);
\draw [fill=black] (-3.236235851472471,3.673856292906182) circle (1.5pt);
\draw [fill=black] (-0.9367416133162614,1.3336091533180803) circle (1.5pt);
\draw [fill=black] (-0.0850770806658131,1.3219661327231145) circle (1.5pt);
\draw [fill=black] (4.956776952624839,3.406066819221971) circle (1.5pt);
\draw [fill=black] (3.1853147247119082,4.954588558352407) circle (1.5pt);
\draw [fill=black] (3.5941137003841233,4.593654919908472) circle (1.5pt);
\draw [fill=black] (4.054012548015366,4.232721281464535) circle (1.5pt);
\draw [fill=black] (1.6523185659411002,1.263751029748286) circle (1.5pt);
\draw [fill=black] (0.7495541613316263,1.2986800915331826) circle (1.5pt);
\draw [fill=black] (-2.3845713188220223,4.3375084668192265) circle (1.5pt);
\draw [fill=black] (2.333650192061459,1.2753940503432515) circle (1.5pt);
\draw [fill=black] (4.496878104993598,3.860144622425633) circle (1.5pt);
\end{scriptsize}
\end{tikzpicture}
\begin{center}
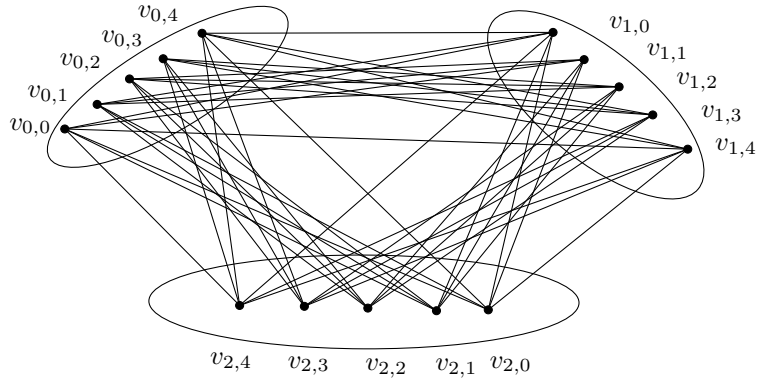
\noindent  Figure 1: In the case when $d=6=2k_1(j-1)+2k_2=(2\times 1)(3-1)+2\times 1$
\end{center}

\noindent  We know that $K_{j \times s}$ consists of $j$ partite sets of size $s$. Given $v_{i,l}$ $0 \leq i \leq j-1$, $0 \leq l \leq s-1$, the set $\{v_{p,r} |$ $ p \neq i$ and $ r\in B_{k_1,s}(l) \}$ will represent the vertices not belonging to the $i^{th}$ partite set (denoted by $V_i$)  that are at most $2k_1$ distance apart inside a partite set (with respect to the second coordinate), as illustrated in the following figure. 

\begin{flushleft}

\begin{tikzpicture}[line cap=round,line join=round,>=triangle 45,x=1.0cm,y=1.0cm]
\clip(-3.1732126760563375,1.2579295194508024) rectangle (7.993812676056322,6.214363386727695);
\draw (1.6233619718309793,4.165191762013734) node[anchor=north west] {$v_{i,l}$};
\draw (1.7545183098591481,5.1001263157894785) node[anchor=north west] {$v_{i,i-k}$};
\draw (-0.00672394366197604,1.8598736842105283) node[anchor=north west] {$V_{i-1}$};
\draw (1.8294647887323874,2.397781235697943) node[anchor=north west] {$v_{i,s-1}$};
\draw (2.073040845070415,5.958216933638449) node[anchor=north west] {$v_{i,0}$};
\draw [rotate around={88.26029569778845:(0.34927183098591585,3.902641647597257)}] (0.34927183098591585,3.902641647597257) ellipse (2.081105001654943cm and 0.9502120759662728cm);
\draw (1.5296788732394302,3.358330434782612) node[anchor=north west] {$v_{i,l+k}$};
\draw [rotate around={88.26029569778893:(2.3540901408450696,3.9026416475972545)}] (2.3540901408450696,3.9026416475972545) ellipse (2.0811050016549366cm and 0.9502120759662758cm);
\draw [rotate around={88.26029569778846:(4.396381690140845,3.877027002288341)}] (4.396381690140845,3.877027002288341) ellipse (2.0811050016549384cm and 0.9502120759662669cm);
\draw [rotate around={88.26029569778846:(6.888352112676057,3.8642196796338717)}] (6.888352112676057,3.8642196796338717) ellipse (2.081105001654923cm and 0.9502120759662585cm);
\draw (2.1292507042253446,1.847066361556066) node[anchor=north west] {$V_{i}$};
\draw (4.1340690140844965,1.8086443935926793) node[anchor=north west] {$V_{i+1}$};
\draw (6.813405633802803,1.795837070938217) node[anchor=north west] {$V_{j-1}$};
\draw [rotate around={88.26029569778893:(-2.158971830985916,3.928256292906184)}] (-2.158971830985916,3.928256292906184) ellipse (2.0811050016549473cm and 0.9502120759662726cm);
\draw (-2.3862746478873245,1.8982956521739152) node[anchor=north west] {$V_0$};
\draw (-2.9109,4.152384439359272) node[anchor=north west] {$v_{0,l}$};
\draw (-2.686060563380282,2.4618178489702545) node[anchor=north west] {$v_{0,s-1}$};
\draw (-2.423747887323944,5.958216933638449) node[anchor=north west] {$v_{0,0}$};
\draw (-0.6437690140845103,4.113962471395885) node[anchor=north west] {$v_{i-1,l}$};
\draw (-0.26903661971831366,2.449010526315792) node[anchor=north west] {$v_{i-1,s-1}$};
\draw (-0.0441971830985957,5.855758352402752) node[anchor=north west] {$v_{i-1,0}$};
\draw (3.422077464788723,4.139577116704809) node[anchor=north west] {$v_{i+1,l}$};
\draw (3.7780732394366097,2.43620320366133) node[anchor=north west] {$v_{i+1,s-1}$};
\draw (4.002912676056328,5.8941803203661385) node[anchor=north west] {$v_{i+1,0}$};
\draw (6.195097183098578,4.126769794050348) node[anchor=north west] {$v_{j-1,l}$};
\draw (6.270043661971817,2.449010526315792) node[anchor=north west] {$v_{j-1,s-1}$};
\draw (6.494883098591536,5.906987643020601) node[anchor=north west] {$v_{j-1,0}$};
\draw (-2.6860605633802823,5.061704347826092)-- (1.08,4.99766773455378);
\draw (1.08,4.99766773455378)-- (1.1362098591549294,4.13957711670481);
\draw (1.1362098591549294,4.13957711670481)-- (-2.6860605633802823,4.177999084668197);
\draw (-2.6860605633802823,4.177999084668197)-- (-2.6860605633802823,5.061704347826092);
\draw (-2.723533802816902,3.5888622425629326)-- (0.930107042253515,3.5888622425629326);
\draw (0.930107042253515,3.5888622425629326)-- (0.930107042253515,2.794808237986273);
\draw (0.930107042253515,2.794808237986273)-- (-2.761007042253522,2.782000915331811);
\draw (-2.761007042253522,2.782000915331811)-- (-2.723533802816902,3.5888622425629326);
\draw (4.209015492957736,5.010475057208243)-- (7.356767605633788,4.959245766590395);
\draw (7.356767605633788,4.959245766590395)-- (7.394240845070408,4.177999084668196);
\draw (7.394240845070408,4.177999084668196)-- (4.265225352112665,4.190806407322658);
\draw (4.265225352112665,4.190806407322658)-- (4.209015492957736,5.010475057208243);
\draw (4.340171830985904,3.601669565217395)-- (7.4317140845070275,3.601669565217395);
\draw (7.4317140845070275,3.601669565217395)-- (7.4317140845070275,2.7179643020594995);
\draw (7.4317140845070275,2.7179643020594995)-- (4.377645070422524,2.692349656750575);
\draw (4.377645070422524,2.692349656750575)-- (4.340171830985904,3.601669565217395);
\draw (-1.8616492957746493,4.9848604118993185) node[anchor=north west] {$ k_1 i \:\: vertices$};
\draw (4.377645070422524,3.4607890160183103) node[anchor=north west] {$(j-1-i)k_1 \:\: vertices$};
\draw (-1.8616492957746493,3.473596338672772) node[anchor=north west] {$ k_1 i\:\: vertices$};
\draw (4.358908450704215,4.843979862700234) node[anchor=north west] {$(j-1-i) k_1 \:\: vertices$};
\begin{scriptsize}
\draw [fill=black] (2.3728267605633726,4.626255377574375) circle (1.5pt);
\draw [fill=black] (2.4290366197183,2.4233958810068676) circle (1.5pt);
\draw [fill=black] (2.4283943661971827,3.9228814645309105) circle (1.5pt);
\draw [fill=black] (2.4290366197183,5.471538672768884) circle (1.5pt);
\draw [fill=black] (2.4103,2.9228814645308954) circle (1.5pt);
\draw [fill=black] (-0.756188732394369,3.8065867276887912) circle (1.5pt);
\draw [fill=black] (-1.0747112676056363,3.8065867276887912) circle (1.5pt);
\draw [fill=black] (-0.9248183098591576,3.8065867276887912) circle (1.5pt);
\draw [fill=black] (2.372826760563372,3.4607890160183103) circle (1.5pt);
\draw [fill=black] (-2.0864887323943675,4.8695945080091585) circle (1.5pt);
\draw [fill=black] (-2.049015492957749,2.449010526315792) circle (1.5pt);
\draw [fill=black] (-2.105225352112677,3.9346599542334135) circle (1.5pt);
\draw [fill=black] (-2.049015492957749,5.509960640732271) circle (1.5pt);
\draw [fill=black] (-2.086488732394368,2.961303432494282) circle (1.5pt);
\draw [fill=black] (-2.0864887323943675,3.4095597254004613) circle (1.5pt);
\draw [fill=black] (0.40548169014084023,4.280457665903895) circle (1.5pt);
\draw [fill=black] (0.40548169014083973,2.3721665903890186) circle (1.5pt);
\draw [fill=black] (0.4429549295774599,3.870623340961102) circle (1.5pt);
\draw [fill=black] (0.31179859154929057,5.39469473684211) circle (1.5pt);
\draw [fill=black] (0.38674507042252987,2.897266819221971) circle (1.5pt);
\draw [fill=black] (4.546274647887313,4.8695945080091585) circle (1.5pt);
\draw [fill=black] (4.565011267605621,2.4105885583524054) circle (1.5pt);
\draw [fill=black] (4.558694366197175,3.883430663615565) circle (1.5pt);
\draw [fill=black] (4.508801408450693,5.407502059496573) circle (1.5pt);
\draw [fill=black] (4.565011267605623,2.8716521739130463) circle (1.5pt);
\draw [fill=black] (4.546274647887313,3.447981693363848) circle (1.5pt);
\draw [fill=black] (7.07571830985914,4.843979862700234) circle (1.5pt);
\draw [fill=black] (7.09445492957745,2.43620320366133) circle (1.5pt);
\draw [fill=black] (7.048309859154933,3.8834306636155627) circle (1.5pt);
\draw [fill=black] (7.038245070422518,5.407502059496573) circle (1.5pt);
\draw [fill=black] (7.05698169014083,2.8460375286041217) circle (1.5pt);
\draw [fill=black] (7.05698169014083,3.4223670480549235) circle (1.5pt);
\draw [fill=black] (5.764154929577449,3.793779405034329) circle (1.5pt);
\draw [fill=black] (5.4456323943661875,3.793779405034329) circle (1.5pt);
\draw [fill=black] (5.595525352112661,3.793779405034329) circle (1.5pt);
\draw [fill=black] (0.3867450704225304,3.4095597254004613) circle (1.5pt);
\draw [fill=black] (-2.0864887323943675,4.267650343249432) circle (1.5pt);
\draw [fill=black] (0.3867450704225304,4.843979862700234) circle (1.5pt);
\draw [fill=black] (4.5275380281690145,4.267650343249432) circle (1.5pt);
\draw [fill=black] (7.03824507042252,4.318879633867281) circle (1.5pt);
\draw [fill=black] (2.3728267605633726,4.267650343249432) circle (1.5pt);
\end{scriptsize}
\end{tikzpicture}
\begin{center}
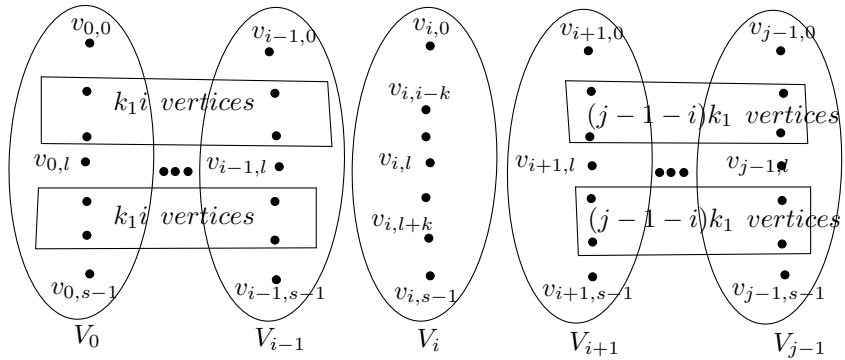
\noindent  Figure 2: The set consisting of $2k_1(j-1)$ vertices corresponding to part(a), namely $\{v_{p,r} | $ $ p \neq i$ and $ r\in B_{k_1,s}(l) \}$ 
\end{center}
\end{flushleft}
\vspace{4pt}

More precisely, it will consist of the vertices

\vspace{4pt}

$v_{0,l-k_1},...,v_{i-1,l-k_1},v_{i+1,l-k_1},...,v_{j-1,l-k_1}$ 

$v_{0,l-k_1+1},...,v_{i-1,l-k_1+1},v_{i+1,l-k_1+1},...,v_{j-1,l-k_1+1}$
 
...

$v_{0,l-1},...,v_{i-1,l-1},v_{i+1,l-1},...,v_{j-1,l-1}$ 

$v_{0,l+1},...,v_{i-1,l+1},v_{i+1,l+1},...,v_{j-1,l+1}$ 

...

$v_{0,l+k_1-1},...,v_{i-1,l+k_1-1},v_{i+1,l+k_1-1},...,v_{j-1,l+k_1-1}$ 

$v_{0,l+k_1},...,v_{i-1,l+k_1},v_{i+1,l+k_1},...,v_{j-1,l+k_1}$ 

\vspace{8pt}

\noindent That is such a set consists of $2k_1(j-1)$ vertices.

\vspace{12pt}

\noindent Similarly, given $v_{i,l}$ where $0 \leq i \leq j-1$ and $0 \leq l \leq s-1$ the set  $\{v_{p,r} |$ $ r \neq l$ and $ p\in B_{k_2,j}(i) \}$ will represent the vertices not belonging to the $i^{th}$ partite set(denoted by $V_i$) that are at most $2k_2$ distance apart between partite sets (with respect to the first  coordinate), as illustrated in the following figure. More precisely, it will consist of the vertices 

\vspace{4pt}

$v_{i-k_2,l},...,v_{i-1,l},v_{i+1,l},...,v_{i+k_2,l}$

\vspace{8pt}

\noindent That is such a set consists of $2k_2$ vertices. 

\begin{flushleft}

\begin{tikzpicture}[line cap=round,line join=round,>=triangle 45,x=1.0cm,y=1.0cm]
\clip(-3.2106859154929577,1.296351487414189) rectangle (7.956339436619702,6.252785354691082);
\draw (1.6233619718309789,4.165191762013734) node[anchor=north west] {$v_{i,l}$};
\draw (1.698308450704218,5.420309382151035) node[anchor=north west] {$v_{i,i-k_1}$};
\draw (-0.006723943661976573,1.8598736842105283) node[anchor=north west] {$V_{i-1}$};
\draw (1.829464788732387,2.397781235697943) node[anchor=north west] {$v_{i,s-1}$};
\draw (2.0730408450704148,5.958216933638449) node[anchor=north west] {$v_{i,0}$};
\draw [rotate around={88.26029569778845:(0.34927183098591585,3.902641647597257)}] (0.34927183098591585,3.902641647597257) ellipse (2.081105001654943cm and 0.9502120759662728cm);
\draw (1.5296788732394295,3.358330434782612) node[anchor=north west] {$v_{i,l+k_1}$};
\draw [rotate around={88.26029569778893:(2.3540901408450696,3.9026416475972545)}] (2.3540901408450696,3.9026416475972545) ellipse (2.0811050016549366cm and 0.9502120759662758cm);
\draw [rotate around={88.26029569778846:(4.396381690140845,3.877027002288341)}] (4.396381690140845,3.877027002288341) ellipse (2.0811050016549384cm and 0.9502120759662669cm);
\draw [rotate around={88.26029569778846:(6.888352112676057,3.8642196796338717)}] (6.888352112676057,3.8642196796338717) ellipse (2.081105001654923cm and 0.9502120759662585cm);
\draw (2.129250704225344,1.847066361556066) node[anchor=north west] {$V_{i}$};
\draw (4.1340690140844965,1.8086443935926793) node[anchor=north west] {$V_{i+1}$};
\draw (6.7946690140844925,1.795837070938217) node[anchor=north west] {$V_{j-1}$};
\draw [rotate around={88.26029569778893:(-2.140235211267607,3.9410636155606453)}] (-2.140235211267607,3.9410636155606453) ellipse (2.081105001654951cm and 0.9502120759662743cm);
\draw (-2.3862746478873254,1.8982956521739152) node[anchor=north west] {$V_0$};
\draw (-2.9109000000000003,4.152384439359272) node[anchor=north west] {$v_{0,l}$};
\draw (-2.854690140845071,5.39469473684211) node[anchor=north west] {$v_{0,i-k_1}$};
\draw (-2.6860605633802823,2.4618178489702545) node[anchor=north west] {$v_{0,s-1}$};
\draw (-2.4237478873239446,5.958216933638449) node[anchor=north west] {$v_{0,0}$};
\draw (-2.9858464788732397,3.3711377574370744) node[anchor=north west] {$v_{0,l+k_1}$};
\draw (-0.3439830985915535,5.39469473684211) node[anchor=north west] {$v_{i-1,i-k_1}$};
\draw (-0.2690366197183142,2.449010526315792) node[anchor=north west] {$v_{i-1,s-1}$};
\draw (0.012012676056333257,5.971024256292912) node[anchor=north west] {$v_{i-1,0}$};
\draw (-0.512612676056342,3.3199084668192254) node[anchor=north west] {$v_{i-1,l+k_1}$};
\draw (3.6656535211267505,5.330658123569799) node[anchor=north west] {$v_{i+1,i-k_1}$};
\draw (3.778073239436609,2.43620320366133) node[anchor=north west] {$v_{i+1,s-1}$};
\draw (3.8904929577464684,5.817336384439365) node[anchor=north west] {$v_{i+1,0}$};
\draw (3.5532338028168913,3.3711377574370744) node[anchor=north west] {$v_{i+1,l+k_1}$};
\draw (6.213833802816888,4.139577116704809) node[anchor=north west] {$v_{j-1,l}$};
\draw (6.1763605633802685,5.369080091533186) node[anchor=north west] {$v_{j-1,i-k_1}$};
\draw (6.401199999999986,2.449010526315792) node[anchor=north west] {$v_{j-1,s-1}$};
\draw (6.607302816901394,5.906987643020601) node[anchor=north west] {$v_{j-1,0}$};
\draw (5.97025774647886,3.4095597254004613) node[anchor=north west] {$v_{j-1,l+k_1}$};
\draw (-1.2058676056338058,4.536604118993139) node[anchor=north west] {$k_2 \:\:vertices$};
\draw (3.590707042253511,4.523796796338678) node[anchor=north west] {$k_2 \:\:vertices$};
\draw (-0.9435549295774681,4.062733180778036)-- (0.5741112676056283,4.062733180778036);
\draw (0.5741112676056283,4.062733180778036)-- (0.5741112676056283,3.691320823798631);
\draw (0.5741112676056283,3.691320823798631)-- (-0.9435549295774681,3.6785135011441685);
\draw (-0.9435549295774681,3.6785135011441685)-- (-0.9435549295774681,4.062733180778036);
\draw (4.171542253521114,4.049925858123572)-- (5.689208450704226,4.049925858123572);
\draw (5.689208450704226,4.049925858123572)-- (5.689208450704226,3.678513501144169);
\draw (5.689208450704226,3.678513501144169)-- (4.171542253521114,3.6657061784897063);
\draw (4.171542253521114,3.6657061784897063)-- (4.171542253521114,4.049925858123572);
\begin{scriptsize}
\draw [fill=black] (2.428394366197174,4.587833409610987) circle (1.5pt);
\draw [fill=black] (2.4290366197183,2.4233958810068676) circle (1.5pt);
\draw [fill=black] (2.4103,3.8962379862700267) circle (1.5pt);
\draw [fill=black] (2.418757746478873,4.973081922196817) circle (1.5pt);
\draw [fill=black] (2.4290366197183,5.471538672768884) circle (1.5pt);
\draw [fill=black] (2.4103,2.9228814645308954) circle (1.5pt);
\draw [fill=black] (-0.756188732394369,3.8065867276887912) circle (1.5pt);
\draw [fill=black] (-0.9248183098591576,3.8065867276887912) circle (1.5pt);
\draw [fill=black] (2.4103,3.550440274599546) circle (1.5pt);
\draw [fill=black] (-2.079529577464788,4.562218764302063) circle (1.5pt);
\draw [fill=black] (-2.049015492957749,2.449010526315792) circle (1.5pt);
\draw [fill=black] (-2.060792957746477,3.9218526315789495) circle (1.5pt);
\draw [fill=black] (-2.070429577464787,4.972053089244854) circle (1.5pt);
\draw [fill=black] (-2.049015492957749,5.509960640732271) circle (1.5pt);
\draw [fill=black] (-2.060792957746477,2.9228814645308945) circle (1.5pt);
\draw [fill=black] (-2.067752112676059,3.5376329519450835) circle (1.5pt);
\draw [fill=black] (0.41137042253520617,4.536604118993138) circle (1.5pt);
\draw [fill=black] (0.4429549295774583,2.4233958810068676) circle (1.5pt);
\draw [fill=black] (0.38674507042252987,3.8834306636155644) circle (1.5pt);
\draw [fill=black] (0.3926338028168963,4.933631121281468) circle (1.5pt);
\draw [fill=black] (0.36800845070422006,5.509960640732271) circle (1.5pt);
\draw [fill=black] (0.43010704225351915,2.871652173913046) circle (1.5pt);
\draw [fill=black] (0.4054816901408386,3.4992109839816967) circle (1.5pt);
\draw [fill=black] (4.564369014084509,4.523796796338676) circle (1.5pt);
\draw [fill=black] (4.565011267605621,2.4105885583524054) circle (1.5pt);
\draw [fill=black] (4.558694366197175,3.883430663615565) circle (1.5pt);
\draw [fill=black] (4.545632394366189,4.895209153318081) circle (1.5pt);
\draw [fill=black] (4.527538028169002,5.343465446224261) circle (1.5pt);
\draw [fill=black] (4.58310563380282,2.8844594965675068) circle (1.5pt);
\draw [fill=black] (4.527538028169001,3.5376329519450835) circle (1.5pt);
\draw [fill=black] (7.029573239436622,4.56221876430206) circle (1.5pt);
\draw [fill=black] (7.038245070422518,2.4746251716247167) circle (1.5pt);
\draw [fill=black] (7.048309859154933,3.8834306636155627) circle (1.5pt);
\draw [fill=black] (7.029573239436622,4.933631121281466) circle (1.5pt);
\draw [fill=black] (7.038245070422518,5.407502059496573) circle (1.5pt);
\draw [fill=black] (7.048309859154933,2.897266819221968) circle (1.5pt);
\draw [fill=black] (7.038245070422518,3.5376329519450835) circle (1.5pt);
\draw [fill=black] (5.782891549295759,3.8193940503432535) circle (1.5pt);
\draw [fill=black] (5.464369014084498,3.8193940503432535) circle (1.5pt);
\draw [fill=black] (5.614261971830971,3.8193940503432535) circle (1.5pt);
\draw [fill=black] (-1.0934478873239468,3.806586727688791) circle (1.5pt);
\end{scriptsize}
\end{tikzpicture}
\begin{center}
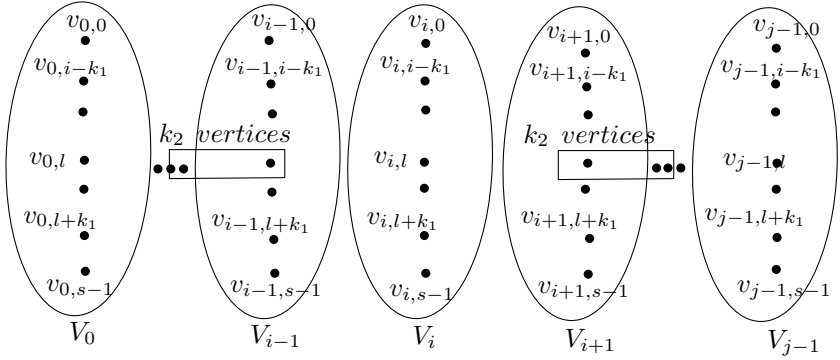
\noindent  Figure 3: The set consisting of $2k_2$  vertices corresponding to part(b), namely $\{v_{p,r} |$ $ r=l$ and $p \in B_{k_2,j}(i) \}$
\end{center}
\end{flushleft}
\vspace{4pt}

\vspace{8pt}
\noindent Thus, by the above definition, part $(a)$ will represent $2k_1(j-1)$ vertices adjacent to $v_{i,l}$ belonging to $V_0,V_1,  V_2,...,V_{i-1},V_{i+1},...V_{j-1}$ and part $(b)$ will represent another  $2k_2$ vertices adjacent to $v_{i,l}$ belonging to $V_{i-k_2},$ $V_{i-k_2+1},...,$ $V_{i-1},V_{i+1},...,V_{i+k_2-1},$ $ V_{i+k_2}$.

\vspace{8pt}
\noindent Therefore, the degree of $v_{i,l}$ will be equal to  $2k_1(j-1)+2k_2$. Also by the remark $l \in B_{k,w}(i)$ if and only if $i \in B_{k,w}(l)$, before the beginning of this section, we get that the graph is well defined.

\vspace{8pt}

\noindent  Next, if $d=(2k_1 + 1)(j-1)+2k_2$ for some non negative integers $k_1$ and $k_2$ such that $2k_1 \le s-3$ and $0<2k_2 \le j-1$. Construct  a graph by connecting the vertices $v_{i,l}$ and $v_{p,r}$ if one of the following situations hold

\vspace{3pt}
\noindent \textbf{a)}  If $r \in B_{k_1,s}(l)$ and $p \neq i$.

\vspace{3pt}

\noindent \textbf{b)} If there exists $w$ such that $4w=(j-1+2k_2)$  and $r \in \sigma_{k_1+1,s}(l)$   and $p \in B_{w,j}(i)$. 

\vspace{3pt}

\noindent \textbf{c)} If there exists $w$ such that  $(j-1+2k_2) -4w =2$  and  $r \in \sigma_{k_1+1,s}(l)$   and $p \in B_{w,j}(i)$ or else $r=l$  and $p \in B_{1,j}(i)$. 

\vspace{12pt}

\begin{tikzpicture}[line cap=round,line join=round,>=triangle 45,x=0.8176752733586924cm,y=0.8056824326247191cm]
\clip(-4.820331882202307,0.28573729977116674) rectangle (6.149107298335471,5.4319524027459964);
\draw (-1.192240973111396,0.7281720823798625) node[anchor=north west] {$v_{2,2}$};
\draw [rotate around={31.31940798555165:(-2.3930879641485308,4.186149199084669)}] (-2.3930879641485308,4.186149199084669) ellipse (1.4876105898931513cm and 0.46655361129097644cm);
\draw [rotate around={-39.71631220657896:(3.8155464788732383,4.087183524027462)}] (3.8155464788732383,4.087183524027462) ellipse (1.4305033028533927cm and 0.5802794837750211cm);
\draw (5.450742381562104,3.743714416475973) node[anchor=north west] {$v_{1,2}$};
\draw (-2.997769782330347,5.466881464530894) node[anchor=north west] {$v_{0,2}$};
\draw (0.42592163892445634,0.7514581235697939) node[anchor=north west] {$v_{2,1}$};
\draw (4.718310883482718,4.5936549199084675) node[anchor=north west] {$v_{1,1}$};
\draw (-3.934600768245841,4.884730434782609) node[anchor=north west] {$v_{0,1}$};
\draw (2.2314504481434074,0.7398151029748282) node[anchor=north west] {$v_{2,0}$};
\draw (3.866646350832269,5.373737299771168) node[anchor=north west] {$v_{1,0}$};
\draw (-4.649998975672218,4.104648054919909) node[anchor=north west] {$v_{ 0,0}$};
\draw [rotate around={-0.24175315415687906:(0.6984542893726006,1.3801812356979415)}] (0.6984542893726006,1.3801812356979415) ellipse (2.3128487440838255cm and 0.5008220121816509cm);
\draw (-0.9367416133162614,1.3336091533180803)-- (-3.236235851472471,3.673856292906182);
\draw (2.333650192061459,1.2753940503432515)-- (-1.4307070422535209,4.942945537757441);
\draw (2.333650192061459,1.2753940503432515)-- (4.956776952624839,3.406066819221971);
\draw (-1.4307070422535209,4.942945537757441)-- (3.1853147247119082,4.954588558352407);
\draw (-3.236235851472471,3.673856292906182)-- (4.956776952624839,3.406066819221971);
\draw (-0.9367416133162614,1.3336091533180803)-- (-2.3845713188220223,4.3375084668192265);
\draw (-0.9367416133162614,1.3336091533180803)-- (3.1853147247119082,4.954588558352407);
\draw (-0.9367416133162614,1.3336091533180803)-- (4.054012548015366,4.232721281464535);
\draw (-2.3845713188220223,4.3375084668192265)-- (3.1853147247119082,4.954588558352407);
\draw (-2.3845713188220223,4.3375084668192265)-- (4.956776952624839,3.406066819221971);
\draw (-2.3845713188220223,4.3375084668192265)-- (2.333650192061459,1.2753940503432515);
\draw (0.7495541613316263,1.2986800915331826)-- (-3.236235851472471,3.673856292906182);
\draw (0.7495541613316263,1.2986800915331826)-- (-1.4307070422535209,4.942945537757441);
\draw (0.7495541613316263,1.2986800915331826)-- (3.1853147247119082,4.954588558352407);
\draw (0.7495541613316263,1.2986800915331826)-- (4.956776952624839,3.406066819221971);
\draw (4.054012548015366,4.232721281464535)-- (-1.4307070422535209,4.942945537757441);
\draw (4.054012548015366,4.232721281464535)-- (-3.236235851472471,3.673856292906182);
\draw (4.054012548015366,4.232721281464535)-- (2.333650192061459,1.2753940503432515);
\begin{scriptsize}
\draw [fill=black] (-1.4307070422535209,4.942945537757441) circle (1.5pt);
\draw [fill=black] (-3.236235851472471,3.673856292906182) circle (1.5pt);
\draw [fill=black] (-0.9367416133162614,1.3336091533180803) circle (1.5pt);
\draw [fill=black] (4.956776952624839,3.406066819221971) circle (1.5pt);
\draw [fill=black] (3.1853147247119082,4.954588558352407) circle (1.5pt);
\draw [fill=black] (4.054012548015366,4.232721281464535) circle (1.5pt);
\draw [fill=black] (0.7495541613316263,1.2986800915331826) circle (1.5pt);
\draw [fill=black] (-2.3845713188220223,4.3375084668192265) circle (1.5pt);
\draw [fill=black] (2.333650192061459,1.2753940503432515) circle (1.5pt);
\end{scriptsize}
\end{tikzpicture}
\begin{center}
\noindent  Figure 4: In the case when $d=4=(2k_1 + 1)(j-1)+2k_2=(2\times 0+1)\times (3-1)+2\times 1$
\end{center}
\vspace{8pt}

\noindent  It should be noted that  the vertex sets of part (b) and part (c) are disjoint and that $j-1+2k_2$ is even as $j$ is odd. Therefore, given $v_{i,l}$, it will be either adjacent the vertices corresponding to part (a) and part (b) or else adjacent the vertices corresponding to part (a) and part (c) according to whether  $4w=(j-1+2k_2)$  or else $(j-1+2k_2) -4w =2$, respectively.  

\vspace{8pt} \noindent As illustrated in figure 2, the set generated by part (a) namely, $\{v_{p,r} |$ $ p \neq i$ and $ r\in B_{k_1,s}(l) \}$  will consist of $2k_1(j-1)$ vertices. 

\vspace{8pt} \noindent Similarly, given $v_{i,l}$ the set generated by part (b), will represent the vertices  belonging to $V_{i-w},...,V_{i-1},V_{i+1},...V_{i+w}$  sets, that are at most $2(k_1+1)$ distance apart between partite sets (with respect to the first coordinate), as illustrated in the following figure. More precisely, it will consist of the vertices

\vspace{4pt}

$v_{i-w,l-(k_1+1)},...,v_{i-1,l-(k_1+1)},v_{i+1,l-(k_1+1)},...v_{i+w,l-(k_1+1)}$ 

$v_{i-w,l+(k_1+1)},...,v_{i-1,l+(k_1+1)},v_{i+1,l+(k_1+1)},...v_{i+w,l+(k_1+1)}$ 

\vspace{4pt}

\noindent Such a set consists of $4w$ vertices. That is, the set consists of $(j-1)+2k_2$ vertices.

\begin{flushleft}

\begin{tikzpicture}[line cap=round,line join=round,>=triangle 45,x=1.0cm,y=1.0cm]
\clip(-3.1732126760563375,1.2579295194508024) rectangle (7.993812676056322,6.214363386727695);
\draw (1.6233619718309793,4.165191762013734) node[anchor=north west] {$v_{i,l}$};
\draw (1.7545183098591481,5.1001263157894785) node[anchor=north west] {$v_{i,i-k-1}$};
\draw (-0.00672394366197604,1.8598736842105283) node[anchor=north west] {$V_{i-w}$};
\draw (1.8294647887323874,2.397781235697943) node[anchor=north west] {$v_{i,s-1}$};
\draw (2.073040845070415,5.958216933638449) node[anchor=north west] {$v_{i,0}$};
\draw [rotate around={88.26029569778845:(-0.15661690140844967,3.9538709382151045)}] (-0.15661690140844967,3.9538709382151045) ellipse (2.0811050016549424cm and 0.9502120759662724cm);
\draw (1.54841549295774,3.358330434782612) node[anchor=north west] {$v_{i,l+k+1}$};
\draw [rotate around={88.26029569778893:(2.3540901408450696,3.9026416475972545)}] (2.3540901408450696,3.9026416475972545) ellipse (2.0811050016549366cm and 0.9502120759662758cm);
\draw [rotate around={88.26029569778846:(4.846060563380282,3.8642196796338757)}] (4.846060563380282,3.8642196796338757) ellipse (2.08110500165495cm and 0.9502120759662721cm);
\draw [rotate around={88.26029569778846:(6.888352112676057,3.8642196796338717)}] (6.888352112676057,3.8642196796338717) ellipse (2.081105001654923cm and 0.9502120759662585cm);
\draw (2.1292507042253446,1.847066361556066) node[anchor=north west] {$V_{i}$};
\draw (4.396381690140834,1.8086443935926793) node[anchor=north west] {$V_{i+w}$};
\draw (6.813405633802803,1.795837070938217) node[anchor=north west] {$V_{j-1}$};
\draw [rotate around={88.26029569778893:(-2.158971830985916,3.928256292906184)}] (-2.158971830985916,3.928256292906184) ellipse (2.0811050016549473cm and 0.9502120759662726cm);
\draw (-2.3862746478873245,1.8982956521739152) node[anchor=north west] {$V_0$};
\draw (-2.9109,4.152384439359272) node[anchor=north west] {$v_{0,l}$};
\draw (-2.686060563380282,2.4618178489702545) node[anchor=north west] {$v_{0,s-1}$};
\draw (-2.423747887323944,5.958216933638449) node[anchor=north west] {$v_{0,0}$};
\draw (-1.018501408450707,4.113962471395885) node[anchor=north west] {$v_{i-w,l}$};
\draw (-0.8123985915492988,2.43620320366133) node[anchor=north west] {$v_{i-w,s-1}$};
\draw (-0.5875591549295808,5.842951029748289) node[anchor=north west] {$v_{i-w,0}$};
\draw (4.902270422535199,4.101155148741423) node[anchor=north west] {$v_{i+w,l}$};
\draw (4.321435211267595,2.43620320366133) node[anchor=north west] {$v_{i+w,s-1}$};
\draw (4.415118309859144,5.881372997711677) node[anchor=north west] {$v_{i+w,0}$};
\draw (6.195097183098578,4.126769794050348) node[anchor=north west] {$v_{j-1,l}$};
\draw (6.401199999999986,2.449010526315792) node[anchor=north west] {$v_{j-1,s-1}$};
\draw (6.607302816901394,5.906987643020601) node[anchor=north west] {$v_{j-1,0}$};
\draw (-0.6250323943661988,5.100126315789479)-- (1.3610492957746476,5.074511670480555);
\draw (1.3610492957746476,5.074511670480555)-- (1.3610492957746474,4.485374828375291);
\draw (1.3610492957746474,4.485374828375291)-- (-0.6812422535211309,4.48537482837529);
\draw (-0.6812422535211309,4.48537482837529)-- (-0.6250323943661988,5.100126315789479);
\draw (-0.7561887323943697,3.0765693363844426)-- (1.1924197183098526,3.089376659038905);
\draw (1.1924197183098526,3.089376659038905)-- (1.2111563380281687,2.5770837528604145);
\draw (1.2111563380281687,2.5770837528604145)-- (-0.7749253521126798,2.55146910755149);
\draw (-0.7749253521126798,2.55146910755149)-- (-0.7561887323943697,3.0765693363844426);
\draw (3.3283943661971835,5.176970251716252)-- (5.501842253521127,5.125740961098404);
\draw (5.501842253521127,5.125740961098404)-- (5.520578873239423,4.60064073226545);
\draw (5.520578873239423,4.60064073226545)-- (3.3658676056338037,4.613448054919912);
\draw (3.3658676056338037,4.613448054919912)-- (3.3283943661971835,5.176970251716252);
\draw (3.459550704225353,3.0765693363844426)-- (5.38942253521127,3.0765693363844426);
\draw (5.38942253521127,3.0765693363844426)-- (5.3706859154929605,2.55146910755149);
\draw (5.3706859154929605,2.55146910755149)-- (3.4782873239436634,2.5258544622425654);
\draw (3.4782873239436634,2.5258544622425654)-- (3.459550704225353,3.0765693363844426);
\draw (-0.4001929577464825,5.151355606407328) node[anchor=north west] {$ w \:\: vertices$};
\draw (3.6469169014084413,2.9869180778032067) node[anchor=north west] {$w \:\: vertices$};
\draw (-0.4376661971831022,3.0253400457665935) node[anchor=north west] {$ w\:\: vertices$};
\draw (3.571970422535202,5.241006864988564) node[anchor=north west] {$w \:\: vertices$};
\begin{scriptsize}
\draw [fill=black] (2.3728267605633726,4.626255377574375) circle (1.5pt);
\draw [fill=black] (2.4290366197183,2.4233958810068676) circle (1.5pt);
\draw [fill=black] (2.4283943661971827,3.9228814645309105) circle (1.5pt);
\draw [fill=black] (2.4290366197183,5.471538672768884) circle (1.5pt);
\draw [fill=black] (2.4103,2.9228814645308954) circle (1.5pt);
\draw [fill=black] (1.2486295774647886,3.883430663615564) circle (1.5pt);
\draw [fill=black] (0.9301070422535155,3.883430663615564) circle (1.5pt);
\draw [fill=black] (1.08,3.883430663615564) circle (1.5pt);
\draw [fill=black] (2.372826760563372,3.4607890160183103) circle (1.5pt);
\draw [fill=black] (-2.1989084507042262,4.741521281464536) circle (1.5pt);
\draw [fill=black] (-2.049015492957749,2.449010526315792) circle (1.5pt);
\draw [fill=black] (-2.105225352112677,3.9346599542334135) circle (1.5pt);
\draw [fill=black] (-2.049015492957749,5.509960640732271) circle (1.5pt);
\draw [fill=black] (-2.086488732394368,2.961303432494282) circle (1.5pt);
\draw [fill=black] (-2.0864887323943675,3.4095597254004613) circle (1.5pt);
\draw [fill=black] (-0.13788028169014488,4.267650343249432) circle (1.5pt);
\draw [fill=black] (-0.13788028169014538,2.359359267734556) circle (1.5pt);
\draw [fill=black] (-0.1004070422535252,3.8578160183066394) circle (1.5pt);
\draw [fill=black] (-0.23156338028169016,5.381887414187648) circle (1.5pt);
\draw [fill=black] (-0.11914366197183503,2.999725400457669) circle (1.5pt);
\draw [fill=black] (4.696167605633792,4.792750572082385) circle (1.5pt);
\draw [fill=black] (4.827323943661959,2.4105885583524054) circle (1.5pt);
\draw [fill=black] (4.8210070422535125,3.883430663615565) circle (1.5pt);
\draw [fill=black] (4.771114084507031,5.407502059496574) circle (1.5pt);
\draw [fill=black] (4.827323943661961,2.9356887871853576) circle (1.5pt);
\draw [fill=black] (4.808587323943651,3.447981693363848) circle (1.5pt);
\draw [fill=black] (7.019508450704211,4.805557894736847) circle (1.5pt);
\draw [fill=black] (7.09445492957745,2.43620320366133) circle (1.5pt);
\draw [fill=black] (7.048309859154933,3.8834306636155627) circle (1.5pt);
\draw [fill=black] (7.038245070422518,5.407502059496573) circle (1.5pt);
\draw [fill=black] (7.07571830985914,2.961303432494282) circle (1.5pt);
\draw [fill=black] (7.05698169014083,3.4223670480549235) circle (1.5pt);
\draw [fill=black] (4.040385915492958,3.793779405034329) circle (1.5pt);
\draw [fill=black] (3.4595507042253524,3.793779405034329) circle (1.5pt);
\draw [fill=black] (3.6094436619718193,3.793779405034329) circle (1.5pt);
\draw [fill=black] (-0.1566169014084547,3.3967524027459985) circle (1.5pt);
\draw [fill=black] (-2.0864887323943675,4.267650343249432) circle (1.5pt);
\draw [fill=black] (-0.13788028169014485,4.664677345537762) circle (1.5pt);
\draw [fill=black] (4.789850704225352,4.267650343249432) circle (1.5pt);
\draw [fill=black] (7.03824507042252,4.318879633867281) circle (1.5pt);
\draw [fill=black] (2.3728267605633726,4.267650343249432) circle (1.5pt);
\end{scriptsize}
\end{tikzpicture}
\begin{center}
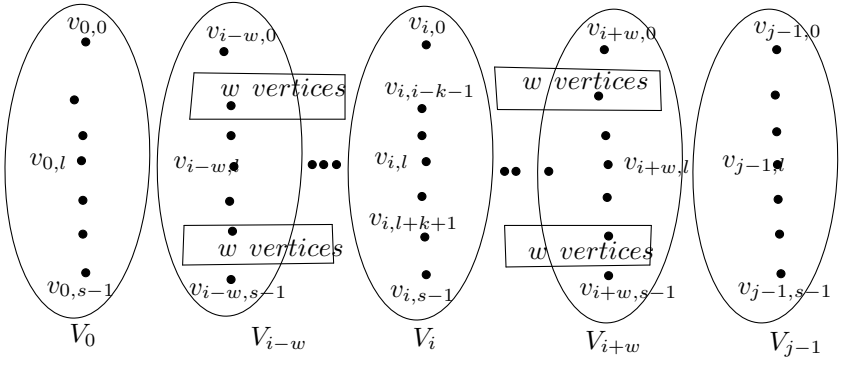
\noindent  Figure 5: The set consisting of $4w$ vertices corresponding to part (b), namely $\{v_{p,r} |$ $ r \in \sigma_{k_1+1,s}(l)$ and $p \in B_{w,j}(i) \}$  
\end{center}
\end{flushleft}
\vspace{4pt}

\noindent Similarly, given $v_{i,l}$ the set generated by the later part (c), will represent the two vertices  belonging to $V_{i-1},V_{i+1}$  sets, namely $v_{i-1,l-1}$, $v_{i+1,l+1}$.  More precisely the set generated by part (c), will consist of the vertices 
\vspace{4pt}

$v_{i-w,l-(k_1-1)},...,v_{i-1,l-(k_1-1)},v_{i+1,l-(k_1-1)},...v_{i+w,l-(k_1-1)}$ 

$v_{i-1,l-1},v_{i+1,l+1}$ 

$v_{i-w,l-(k_1+1)},...,v_{i-1,l-(k_1+1)},v_{i+1,l-(k_1+1)},...v_{i+w,l-(k_1+1)}$ 

\vspace{14pt}

\noindent Such a set consists of $4w+2$ vertices. That is, the set consists of $(j-1)+2k_2$ vertices.

\begin{flushleft}
\begin{tikzpicture}[line cap=round,line join=round,>=triangle 45,x=1.0cm,y=1.0cm]
\clip(-3.1732126760563375,1.2579295194508024) rectangle (7.993812676056322,6.214363386727695);
\draw (1.6233619718309793,4.165191762013734) node[anchor=north west] {$v_{i,l}$};
\draw (-0.00672394366197604,1.8598736842105283) node[anchor=north west] {$V_{i-1}$};
\draw (1.773254929577458,2.397781235697943) node[anchor=north west] {$v_{i,s-1}$};
\draw (2.073040845070415,5.958216933638449) node[anchor=north west] {$v_{i,0}$};
\draw [rotate around={88.26029569778845:(0.34927183098591585,3.902641647597257)}] (0.34927183098591585,3.902641647597257) ellipse (2.081105001654943cm and 0.9502120759662728cm);
\draw [rotate around={88.26029569778893:(2.3540901408450696,3.9026416475972545)}] (2.3540901408450696,3.9026416475972545) ellipse (2.0811050016549366cm and 0.9502120759662758cm);
\draw [rotate around={88.26029569778846:(4.396381690140845,3.877027002288341)}] (4.396381690140845,3.877027002288341) ellipse (2.0811050016549384cm and 0.9502120759662669cm);
\draw [rotate around={88.26029569778846:(6.888352112676057,3.8642196796338717)}] (6.888352112676057,3.8642196796338717) ellipse (2.081105001654923cm and 0.9502120759662585cm);
\draw (2.1292507042253446,1.847066361556066) node[anchor=north west] {$V_{i}$};
\draw (4.1340690140844965,1.8086443935926793) node[anchor=north west] {$V_{i+1}$};
\draw (6.813405633802803,1.795837070938217) node[anchor=north west] {$V_{j-1}$};
\draw [rotate around={88.26029569778893:(-2.158971830985916,3.928256292906184)}] (-2.158971830985916,3.928256292906184) ellipse (2.0811050016549473cm and 0.9502120759662726cm);
\draw (-2.3862746478873245,1.8982956521739152) node[anchor=north west] {$V_0$};
\draw (-2.9109,4.152384439359272) node[anchor=north west] {$v_{0,l}$};
\draw (-2.686060563380282,2.4618178489702545) node[anchor=north west] {$v_{0,s-1}$};
\draw (-2.423747887323944,5.958216933638449) node[anchor=north west] {$v_{0,0}$};
\draw (-0.7749253521126791,4.203613729977121) node[anchor=north west] {$v_{i-1,l}$};
\draw (-0.26903661971831366,2.449010526315792) node[anchor=north west] {$v_{i-1,s-1}$};
\draw (-0.0441971830985957,5.855758352402752) node[anchor=north west] {$v_{i-1,0}$};
\draw (3.347130985915484,4.165191762013734) node[anchor=north west] {$v_{i+1,l}$};
\draw (3.7780732394366097,2.43620320366133) node[anchor=north west] {$v_{i+1,s-1}$};
\draw (3.9279661971830886,5.868565675057214) node[anchor=north west] {$v_{i+1,0}$};
\draw (5.914047887323931,4.126769794050348) node[anchor=north west] {$v_{j-1,l}$};
\draw (6.213833802816888,2.4233958810068676) node[anchor=north west] {$v_{j-1,s-1}$};
\draw (6.457409859154915,5.8941803203661385) node[anchor=north west] {$v_{j-1,0}$};
\draw (0.14316901408450639,4.075540503432498)-- (0.5179014084506997,4.062733180778034);
\draw (0.5179014084506997,4.062733180778034)-- (0.5741112676056346,3.691320823798631);
\draw (0.5741112676056346,3.691320823798631)-- (0.16190563380281614,3.691320823798631);
\draw (0.16190563380281614,3.691320823798631)-- (0.14316901408450639,4.075540503432498);
\draw (-0.08167042253521536,4.498182151029753) node[anchor=north west] {$ 1 \: vertex$};
\draw (3.8155464788732294,4.408530892448517) node[anchor=north west] {$ 1 \: vertex$};
\draw (4.321435211267606,3.99869656750572)-- (4.6961676056338035,3.9858892448512617);
\draw (4.6961676056338035,3.9858892448512617)-- (4.7523774647887205,3.614476887871857);
\draw (4.7523774647887205,3.614476887871857)-- (4.340171830985916,3.614476887871857);
\draw (4.340171830985916,3.614476887871857)-- (4.321435211267606,3.99869656750572);
\begin{scriptsize}
\draw [fill=black] (2.3728267605633726,4.626255377574375) circle (1.5pt);
\draw [fill=black] (2.4290366197183,2.4233958810068676) circle (1.5pt);
\draw [fill=black] (2.3728267605633726,3.8962379862700267) circle (1.5pt);
\draw [fill=black] (2.4290366197183,5.471538672768884) circle (1.5pt);
\draw [fill=black] (2.3915633802816822,3.153413272311216) circle (1.5pt);
\draw [fill=black] (-0.756188732394369,3.8065867276887912) circle (1.5pt);
\draw [fill=black] (-1.0747112676056363,3.8065867276887912) circle (1.5pt);
\draw [fill=black] (-0.9248183098591576,3.8065867276887912) circle (1.5pt);
\draw [fill=black] (2.372826760563372,3.4607890160183103) circle (1.5pt);
\draw [fill=black] (-2.142698591549297,4.728713958810073) circle (1.5pt);
\draw [fill=black] (-2.049015492957749,2.449010526315792) circle (1.5pt);
\draw [fill=black] (-2.105225352112677,3.9346599542334135) circle (1.5pt);
\draw [fill=black] (-2.049015492957749,5.509960640732271) circle (1.5pt);
\draw [fill=black] (-2.0864887323943675,3.089376659038905) circle (1.5pt);
\draw [fill=black] (-2.0864887323943675,3.4095597254004613) circle (1.5pt);
\draw [fill=black] (0.4242183098591501,4.421338215102979) circle (1.5pt);
\draw [fill=black] (0.40548169014083973,2.3721665903890186) circle (1.5pt);
\draw [fill=black] (0.34927183098591075,3.8834306636155644) circle (1.5pt);
\draw [fill=black] (0.3867450704225304,5.420309382151035) circle (1.5pt);
\draw [fill=black] (0.3867450704225304,3.050954691075518) circle (1.5pt);
\draw [fill=black] (4.527538028169003,4.690291990846687) circle (1.5pt);
\draw [fill=black] (4.490064788732384,2.4105885583524054) circle (1.5pt);
\draw [fill=black] (4.508801408450693,5.407502059496573) circle (1.5pt);
\draw [fill=black] (4.527538028169003,3.0253400457665935) circle (1.5pt);
\draw [fill=black] (4.508801408450693,3.4223670480549235) circle (1.5pt);
\draw [fill=black] (7.03824507042252,4.639062700228838) circle (1.5pt);
\draw [fill=black] (7.03824507042252,2.449010526315792) circle (1.5pt);
\draw [fill=black] (7.048309859154933,3.8834306636155627) circle (1.5pt);
\draw [fill=black] (7.019508450704211,5.458731350114422) circle (1.5pt);
\draw [fill=black] (7.07571830985914,3.0637620137299804) circle (1.5pt);
\draw [fill=black] (7.05698169014083,3.4223670480549235) circle (1.5pt);
\draw [fill=black] (5.764154929577449,3.793779405034329) circle (1.5pt);
\draw [fill=black] (5.4456323943661875,3.793779405034329) circle (1.5pt);
\draw [fill=black] (5.595525352112661,3.793779405034329) circle (1.5pt);
\draw [fill=black] (0.3867450704225304,3.4095597254004613) circle (1.5pt);
\draw [fill=black] (-2.105225352112677,4.408530892448517) circle (1.5pt);
\draw [fill=black] (0.40548169014084023,4.728713958810073) circle (1.5pt);
\draw [fill=black] (4.527538028169003,4.395723569794055) circle (1.5pt);
\draw [fill=black] (7.03824507042252,4.318879633867281) circle (1.5pt);
\draw [fill=black] (2.3728267605633726,4.382916247139592) circle (1.5pt);
\draw [fill=black] (4.527538028169017,3.806586727688791) circle (1.5pt);
\end{scriptsize}
\end{tikzpicture}
\begin{center}
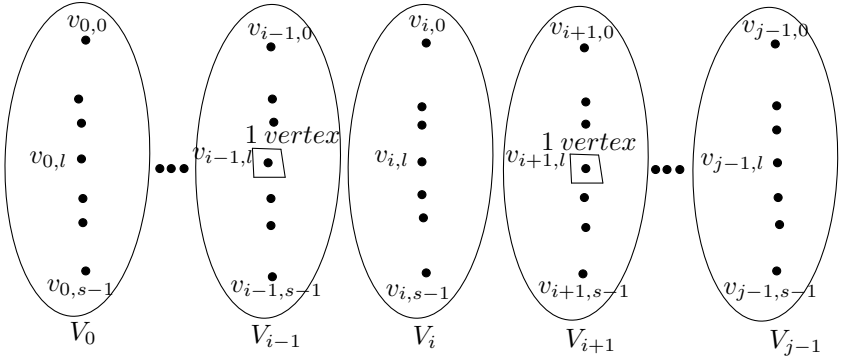
\noindent  Figure 6: The set consisting of $2$ vertices corresponding to later section of part (c), namely  $\{v_{p,r} |$ $ p \in B_{w,j}(i)$ or else $r=l$  and $p \in B_{1,j}(i) \}$ 
\end{center}
\end{flushleft}
\vspace{4pt}

\vspace{8pt}
\noindent Therefore, the degree of $v_{i,l}$ will be equal to  $2k_1(j-1)+2k_2$ when part (a)+(b) situation arrises or when part (a)+(c) situation arrises.
\end{proof}

\begin{lemma}
\label{lemma 2}
There exist some uniform graphs of degree $d$ on $V(K_{j \times s})$ if $j$ is even or $s$ is even.
\end{lemma}

\begin{proof} 
We approach this problem by considering the following three cases.

\vspace{4pt}

\noindent \textbf{Case 1)} If $j$ is even and $s$ is odd.

\noindent  If $d=2k_1(j-1)+2k_2$ for some non negative integers $k_1$ and $k_2$ such that $2k_1 \le s-1$ and $0<2k_2 \le j-2$. Construct  a graph by connecting the vertices $v_{i,l}$ and $v_{p,r}$ if one of the following situations hold

\vspace{4pt}
\noindent \textbf{a)} If $r \in B_{k_1,s}(l)$ and $p \neq i$. \hspace{14pt}\textbf{b)}  If $r=l$ and $p \in B_{k_2,j}(i)$.

\vspace{4pt}

\noindent The vertex $v_{i,l}$ will be either adjacent the vertices corresponding to part (a) or part (b) and they are respectively equal to $2k_1(j -1)$ and $2k_2$. Therefore, we get that the degree of $v_{i,l}$ is equal to $2k_1(j-1)+2k_2$ as required.

\vspace{9pt}
\noindent If $d=2k_1(j-1)+2k_2+1$ for some non negative integers $k_1$ and $k_2$ such that $2k_1 \le s-1$ and $ 2k_2\le j-2$. Construct  a graph by connecting the vertices $v_{i,l}$ and $v_{p,r}$ if one of the following situations hold

\vspace{4pt}
\noindent \textbf{a)}  If $r \in B_{k_1,s}(l)$ and $p \neq i$. \hspace{14pt} \textbf{b)}  If $r=l$ and $p \in B_{k_2,j}(i)$.

\vspace{4pt}
\noindent \textbf{c)}  If $r=l$ and $p \in \sigma_{\frac{j}{2},j}(i)$.

\begin{tikzpicture}[line cap=round,line join=round,>=triangle 45,x=1.0cm,y=1.0cm]
\clip(-4.908904993597956,0.6967359267734548) rectangle (5.021503457106274,6.832607780320368);
\draw [rotate around={87.91580736590971:(-3.3954971190781054,3.661631121281466)}] (-3.3954971190781054,3.661631121281466) ellipse (1.547757610453977cm and 0.8582776916516105cm);
\draw [rotate around={-89.17565717915494:(3.254711267605635,3.582053089244852)}] (3.254711267605635,3.582053089244852) ellipse (1.5866417120357312cm and 0.7648084220062353cm);
\draw (-3.9380074263764446,2.873980778032037) node[anchor=north west] {$v_{0,0}$};
\draw [rotate around={0.0:(0.20363719590267826,5.790557437070943)}] (0.20363719590267826,5.790557437070943) ellipse (2.3718147816834696cm and 0.5206116257827407cm);
\draw [rotate around={-0.15358490033993571:(0.33820019206145846,1.5915020594965694)}] (0.33820019206145846,1.5915020594965694) ellipse (2.437053978284441cm and 0.4819305243903091cm);
\draw (-1.468180281690144,6.820964759725403) node[anchor=north west] {$v_{1,0}$};
\draw (4.016539308578745,4.993010526315791) node[anchor=north west] {$v_{2,0}$};
\draw (-3.9380074263764446,4.876580320366133) node[anchor=north west] {$v_{0,2}$};
\draw (1.6318786171574886,6.797678718535471) node[anchor=north west] {$v_{1,2}$};
\draw (4.1357723431498075,2.75755057208238) node[anchor=north west] {$v_{2,2}$};
\draw (-4.023173879641489,3.921852631578948) node[anchor=north west] {$v_{0,1}$};
\draw (4.152805633802816,3.8287084668192226) node[anchor=north west] {$v_{2,1}$};
\draw (1.5978120358514707,1.0925986270022878) node[anchor=north west] {$v_{3,0}$};
\draw (0.04778258642765439,1.0809556064073222) node[anchor=north west] {$v_{3,1}$};
\draw (-1.4511469910371348,1.0925986270022878) node[anchor=north west] {$v_{3,2}$};
\draw (-3.0607929577464787,3.652898855835243)-- (3.1972380281690143,3.6272842105263186);
\draw (0.14316901408450813,5.9197949656750595)-- (0.3305352112676062,1.3988100686498866);
\draw (-3.0607929577464787,3.652898855835243)-- (0.14316901408450813,5.9197949656750595);
\draw (-3.0607929577464787,3.652898855835243)-- (0.3305352112676062,1.3988100686498866);
\draw (3.1972380281690143,3.6272842105263186)-- (0.14316901408450813,5.9197949656750595);
\draw (3.1972380281690143,3.6272842105263186)-- (0.3305352112676062,1.3988100686498866);
\draw (-3.06,2.66)-- (-1.262077464788731,5.932602288329522);
\draw (-1.262077464788731,5.932602288329522)-- (3.1597647887323945,4.626255377574374);
\draw (3.1597647887323945,4.626255377574374)-- (1.8669380281690149,1.4116173913043493);
\draw (-3.06,2.66)-- (1.8669380281690149,1.4116173913043493);
\draw (-3.02,4.58)-- (1.7919915492957754,5.958216933638447);
\draw (1.7919915492957754,5.958216933638447)-- (3.234711267605634,2.5770837528604145);
\draw (3.234711267605634,2.5770837528604145)-- (-1.0747112676056336,1.4116173913043493);
\draw (-1.0747112676056336,1.4116173913043493)-- (-3.02,4.58);
\draw (-3.06,2.66)-- (3.1597647887323945,4.626255377574374);
\draw (-1.262077464788731,5.932602288329522)-- (1.8669380281690149,1.4116173913043493);
\draw (-3.02,4.58)-- (3.234711267605634,2.5770837528604145);
\draw (1.7919915492957754,5.958216933638447)-- (-1.0747112676056336,1.4116173913043493);
\draw (-4.568239180537776,0.4988045766590382) node[anchor=north west] {$Example: When$  $d=3=2\times 0\times (4-1)+2 \times 1 +1$};
\draw (-0.05441715749039942,6.855893821510299) node[anchor=north west] {$v_{1,1}$};
\begin{scriptsize}
\draw [fill=black] (-3.02,4.58) circle (1.5pt);
\draw [fill=black] (-3.0607929577464787,3.652898855835243) circle (1.5pt);
\draw [fill=black] (-3.06,2.66) circle (1.5pt);
\draw [fill=black] (3.1597647887323945,4.626255377574374) circle (1.5pt);
\draw [fill=black] (3.1972380281690143,3.6272842105263186) circle (1.5pt);
\draw [fill=black] (3.234711267605634,2.5770837528604145) circle (1.5pt);
\draw [fill=black] (-1.262077464788731,5.932602288329522) circle (1.5pt);
\draw [fill=black] (0.14316901408450813,5.9197949656750595) circle (1.5pt);
\draw [fill=black] (1.7919915492957754,5.958216933638447) circle (1.5pt);
\draw [fill=black] (-1.0747112676056336,1.4116173913043493) circle (1.5pt);
\draw [fill=black] (0.3305352112676062,1.3988100686498866) circle (1.5pt);
\draw [fill=black] (1.8669380281690149,1.4116173913043493) circle (1.5pt);
\end{scriptsize}
\end{tikzpicture}
\begin{center}
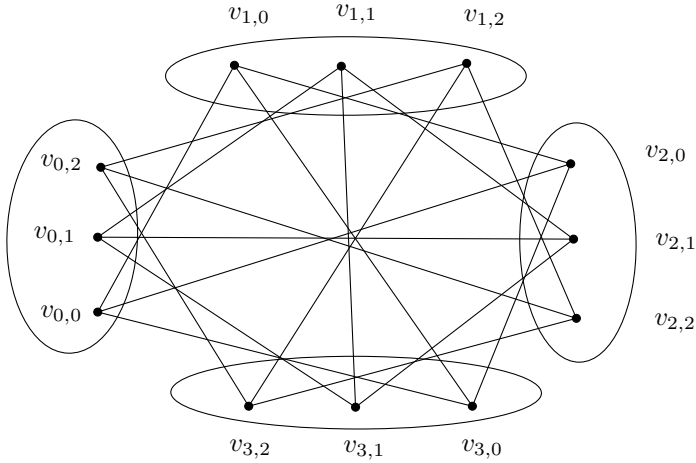
\noindent  Figure 7: In the case(1) when $d=3=2k_1(j-1)+2k_2+1=2\times 0\times (4-1)+2\times 1+1$
\end{center}
\vspace{4pt}

\vspace{8pt}

\noindent The vertex $v_{i,l}$  will be either adjacent the vertices corresponding to part
(a), part (b) or part (c) and they are respectively equal to $2k_1(j-1)$, $ $ 
$2k_2$  and one. Therefore, we get that the degree of $v_{i,l}$ is equal to $2k_1(j-1)+
2k_2 + 1$ as required.

\vspace{12pt}

\noindent If $d=(2k_1+1)(j-1)+m$ for some non negative integers $k_1$, $k_2$ and $m$ such that $2k_1 \le s-3$ and $0<m \le j-1$ where $m=2k_2$ or $m=2k_2+1$. Construct  a graph by connecting the vertices $v_{i,l}$ and $v_{p,r}$ if one of the following situations hold

\noindent \textbf{a)} If $r \in B_{k_1,s}(l)$ and $p \neq i$.

\noindent \textbf{b)} If there exists $w$ such that  $w=(j-1+m) $ div $ 4$, $r \in \sigma_{k_1+1,s}(l)$   and $p \in B_{w,j}(i)$.

\noindent \textbf{c)} If there exists $w$ such that  $(j-1+m) -4w =1$,   $r=l$ and $p \in \sigma_{\frac{j}{2},j}(i)$.

\noindent \textbf{d)}  If there exists $w$ such that  $(j-1+m) -4w =2$,   $r=l$ and $p \in \sigma_{1,j}(i)$.

\noindent \textbf{e)} If there exists $w$ such that  $(j-1+m) -4w =3$,   $r=l$, $p \in \sigma_{1,j}(i)$ and $p \in \sigma_{\frac{j}{2},j}(i)$ (as $j$ is even). 

\begin{tikzpicture}[line cap=round,line join=round,>=triangle 45,x=1.0cm,y=1.0cm]
\clip(-4.56483252240717,0.3788814645308922) rectangle (5.484808962868118,6.165462700228835);
\draw (4.769410755441742,4.10464805491991) node[anchor=north west] {$v_{2,2}$};
\draw (1.9929843790012807,6.165462700228835) node[anchor=north west] {$v_{1,2}$};
\draw (-3.014803072983354,5.373737299771169) node[anchor=north west] {$v_{0,2}$};
\draw (4.5990778489116515,4.686799084668194) node[anchor=north west] {$v_{2,1}$};
\draw (0.8006540332906532,6.188748741418767) node[anchor=north west] {$v_{1,1}$};
\draw (-3.662068117797695,4.7799432494279195) node[anchor=north west] {$v_{0,1}$};
\draw (4.019945966709347,5.164162929061787) node[anchor=north west] {$v_{2,0}$};
\draw (-0.2894765685019204,6.1538196796338696) node[anchor=north west] {$v_{1,0}$};
\draw (-4.190100128040973,4.116291075514876) node[anchor=north west] {$v_{ 0,0}$};
\draw [rotate around={0.19292657820113565:(1.0476367477592723,5.437773913043484)}] (1.0476367477592723,5.437773913043484) ellipse (1.7483499660249686cm and 0.2601367191474572cm);
\draw [rotate around={44.34721839696155:(-2.265338284250954,4.273471853546919)}] (-2.265338284250954,4.273471853546919) ellipse (1.277571838221598cm and 0.4624295414284397cm);
\draw [rotate around={-46.3898171681946:(3.8836796414852803,4.250185812356978)}] (3.8836796414852803,4.250185812356978) ellipse (1.1341141027908277cm and 0.45875593740778126cm);
\draw [rotate around={-37.589712097674635:(-2.2568216389244564,2.2766938215103014)}] (-2.2568216389244564,2.2766938215103014) ellipse (1.2147938440082502cm and 0.5001952216294918cm);
\draw [rotate around={-0.21055950044071972:(0.8176873239436626,1.3976457665903907)}] (0.8176873239436626,1.3976457665903907) ellipse (1.6472027901808877cm and 0.4515339493587839cm);
\draw [rotate around={50.49894937040257:(4.0625291933418675,2.241764759725401)}] (4.0625291933418675,2.241764759725401) ellipse (1.0566703694701123cm and 0.46138358441007754cm);
\draw (-1.856539308578745,4.814872311212819)-- (1.9759510883482743,5.431952402745995);
\draw (1.9759510883482743,5.431952402745995)-- (4.241378745198464,3.8717876430205993);
\draw (4.241378745198464,3.8717876430205993)-- (3.696313444302177,1.8109729977116729);
\draw (3.696313444302177,1.8109729977116729)-- (-0.20431011523687556,1.2986800915331829);
\draw (-0.20431011523687556,1.2986800915331829)-- (-2.793370294494238,2.695842562929065);
\draw (-2.793370294494238,2.695842562929065)-- (-1.856539308578745,4.814872311212819);
\draw (-2.3845713188220223,4.3375084668192265)-- (0.9880202304737525,5.443595423340967);
\draw (0.9880202304737525,5.443595423340967)-- (3.9518128040973113,4.290936384439364);
\draw (3.9518128040973113,4.290936384439364)-- (4.071045838668374,2.2184787185354717);
\draw (4.071045838668374,2.2184787185354717)-- (0.7495541613316263,1.2986800915331826);
\draw (0.7495541613316263,1.2986800915331826)-- (-2.26533828425096,2.2766938215103005);
\draw (-2.26533828425096,2.2766938215103005)-- (-2.3845713188220223,4.3375084668192265);
\draw (-2.8615034571062736,3.801929519450805)-- (0.03415595390524991,5.443595423340967);
\draw (0.03415595390524991,5.443595423340967)-- (3.4408140845070423,4.710085125858129);
\draw (3.4408140845070423,4.710085125858129)-- (4.394678361075544,2.6841995423340994);
\draw (4.394678361075544,2.6841995423340994)-- (1.7374850192061462,1.287037070938217);
\draw (1.7374850192061462,1.287037070938217)-- (-1.7202729833546733,1.84590205949657);
\draw (-1.7202729833546733,1.84590205949657)-- (-2.8615034571062736,3.801929519450805);
\draw (-2.3845713188220223,4.3375084668192265)-- (-2.793370294494238,2.695842562929065);
\draw (-2.3845713188220223,4.3375084668192265)-- (0.03415595390524991,5.443595423340967);
\draw (-2.3845713188220223,4.3375084668192265)-- (1.9759510883482743,5.431952402745995);
\draw (-2.3845713188220223,4.3375084668192265)-- (-1.7202729833546733,1.84590205949657);
\draw (-1.856539308578745,4.814872311212819)-- (0.03415595390524991,5.443595423340967);
\draw (-1.856539308578745,4.814872311212819)-- (0.9880202304737525,5.443595423340967);
\draw (-1.856539308578745,4.814872311212819)-- (-2.26533828425096,2.2766938215103005);
\draw (-1.856539308578745,4.814872311212819)-- (-1.7202729833546733,1.84590205949657);
\draw (-2.8615034571062736,3.801929519450805)-- (-2.793370294494238,2.695842562929065);
\draw (-2.8615034571062736,3.801929519450805)-- (-2.26533828425096,2.2766938215103005);
\draw (-2.8615034571062736,3.801929519450805)-- (0.9880202304737525,5.443595423340967);
\draw (-2.8615034571062736,3.801929519450805)-- (1.9759510883482743,5.431952402745995);
\draw (0.03415595390524991,5.443595423340967)-- (3.9518128040973113,4.290936384439364);
\draw (0.03415595390524991,5.443595423340967)-- (4.241378745198464,3.8717876430205993);
\draw (0.9880202304737525,5.443595423340967)-- (3.4408140845070423,4.710085125858129);
\draw (0.9880202304737525,5.443595423340967)-- (4.241378745198464,3.8717876430205993);
\draw (1.9759510883482743,5.431952402745995)-- (3.4408140845070423,4.710085125858129);
\draw (1.9759510883482743,5.431952402745995)-- (3.9518128040973113,4.290936384439364);
\draw (3.4408140845070423,4.710085125858129)-- (4.071045838668374,2.2184787185354717);
\draw (3.4408140845070423,4.710085125858129)-- (3.696313444302177,1.8109729977116729);
\draw (3.9518128040973113,4.290936384439364)-- (4.394678361075544,2.6841995423340994);
\draw (3.9518128040973113,4.290936384439364)-- (3.696313444302177,1.8109729977116729);
\draw (4.241378745198464,3.8717876430205993)-- (4.394678361075544,2.6841995423340994);
\draw (4.241378745198464,3.8717876430205993)-- (4.071045838668374,2.2184787185354717);
\draw (4.394678361075544,2.6841995423340994)-- (0.7495541613316263,1.2986800915331826);
\draw (4.394678361075544,2.6841995423340994)-- (-0.20431011523687556,1.2986800915331829);
\draw (4.071045838668374,2.2184787185354717)-- (1.7374850192061462,1.287037070938217);
\draw (4.071045838668374,2.2184787185354717)-- (-0.20431011523687556,1.2986800915331829);
\draw (3.696313444302177,1.8109729977116729)-- (1.7374850192061462,1.287037070938217);
\draw (3.696313444302177,1.8109729977116729)-- (0.7495541613316263,1.2986800915331826);
\draw (1.7374850192061462,1.287037070938217)-- (-2.26533828425096,2.2766938215103005);
\draw (1.7374850192061462,1.287037070938217)-- (-2.793370294494238,2.695842562929065);
\draw (0.7495541613316263,1.2986800915331826)-- (-1.7202729833546733,1.84590205949657);
\draw (0.7495541613316263,1.2986800915331826)-- (-2.793370294494238,2.695842562929065);
\draw (-0.20431011523687556,1.2986800915331829)-- (-1.7202729833546733,1.84590205949657);
\draw (-0.20431011523687556,1.2986800915331829)-- (-2.26533828425096,2.2766938215103005);
\draw (4.905677080665813,2.9054169336384446) node[anchor=north west] {$v_{3,0}$};
\draw (4.7012775928297055,2.393124027459955) node[anchor=north west] {$v_{3,1}$};
\draw (4.2924786171574905,1.9274032036613276) node[anchor=north west] {$v_{3,2}$};
\draw (1.5671521126760566,0.8562453089244851) node[anchor=north west] {$v_{4,0}$};
\draw (0.6303211267605636,0.8329592677345538) node[anchor=north west] {$v_{4,1}$};
\draw (-0.442776184379001,0.8678883295194508) node[anchor=north west] {$v_{4,2}$};
\draw (-3.4065687580025603,1.9274032036613276) node[anchor=north west] {$v_{5,0}$};
\draw (-3.9005341869398205,2.3232659038901606) node[anchor=north west] {$v_{5,1}$};
\draw (-4.258233290653009,2.893773913043479) node[anchor=north west] {$v_{5,2}$};
\begin{scriptsize}
\draw [fill=black] (-1.856539308578745,4.814872311212819) circle (1.5pt);
\draw [fill=black] (-2.8615034571062736,3.801929519450805) circle (1.5pt);
\draw [fill=black] (-0.20431011523687556,1.2986800915331829) circle (1.5pt);
\draw [fill=black] (4.241378745198464,3.8717876430205993) circle (1.5pt);
\draw [fill=black] (3.4408140845070423,4.710085125858129) circle (1.5pt);
\draw [fill=black] (3.9518128040973113,4.290936384439364) circle (1.5pt);
\draw [fill=black] (0.7495541613316263,1.2986800915331826) circle (1.5pt);
\draw [fill=black] (-2.3845713188220223,4.3375084668192265) circle (1.5pt);
\draw [fill=black] (1.7374850192061462,1.287037070938217) circle (1.5pt);
\draw [fill=black] (0.03415595390524991,5.443595423340967) circle (1.5pt);
\draw [fill=black] (0.9880202304737525,5.443595423340967) circle (1.5pt);
\draw [fill=black] (1.9759510883482743,5.431952402745995) circle (1.5pt);
\draw [fill=black] (-2.793370294494238,2.695842562929065) circle (1.5pt);
\draw [fill=black] (-2.26533828425096,2.2766938215103005) circle (1.5pt);
\draw [fill=black] (-1.7202729833546733,1.84590205949657) circle (1.5pt);
\draw [fill=black] (4.394678361075544,2.6841995423340994) circle (1.5pt);
\draw [fill=black] (4.071045838668374,2.2184787185354717) circle (1.5pt);
\draw [fill=black] (3.696313444302177,1.8109729977116729) circle (1.5pt);
\end{scriptsize}
\end{tikzpicture}
\begin{center}
\noindent  Figure 8: In the case(1) when  $d=6=(2k_1+1)(j-1)+2k_2+1=(2\times 0 +1)(6-1)+2\times 0 +1$
\end{center}

\vspace{8pt}
\noindent It should be noted that the vertex sets of part (b), (c), (d) and part
(e) are disjoint. Therefore, $v_{i,l}$ will be either adjacent the vertices
corresponding to part (a) and part (b) or (a) and part (c) or (a) and part (d)
or else part (a) and part (e) according
to whether $4w=j-1+2k_2$, $(j-1+2k_2)-4w=1$, $(j-1+2k_2)-4w=2$ or $(j-1+2k_2)-4w=3$  respectively. As done before, in all these scenarios
we get $d=(2k_1+1)(j-1)+m$ as required.

\vspace{10pt}
\noindent \textbf{Case 2)} If $j$ is even and $s$ is even.

\vspace{8pt}

\noindent  If $d=2k_1(j-1)+2k_2$ for some non negative integers $k_1$ and $k_2$ such that $2k_1 \le s-2$ and $0<k_2 \le j-1$. Construct  a graph by connecting the vertices $v_{i,l}$ and $v_{p,r}$ if one of the following situations hold

\vspace{4pt}
\noindent \textbf{a)}  If $r \in B_{k_1,s}(l)$ and $p \neq i$.

\vspace{4pt}
\noindent \textbf{b)}  If $k_2 < \frac{j}{2}$, $r=l$ and $p \in B_{k_2,j}(i)$.

\vspace{4pt}

\noindent \textbf{c)}  If $\frac{j}{2} \le k_2 <  j-1$ and  (($r=l$ and $p \in B_{\frac{j-2}{2},j}(i)$) or ($r \in \sigma_{\frac{s}{2},s}(l)$  and      $p \in B_{\frac{2k_2-(j-2)}{2},j}(i)$)).

\vspace{4pt}
\noindent \textbf{d)}  If $k_2 =j-1$ and  (($r=l$ and $p\neq i$) or ($r \in \sigma_{\frac{s}{2},s}(l)$ and $p\neq i$)).

\vspace{10pt}
\noindent It should be noted that the vertex sets of part (b), (c) and part (d)
are disjoint. Therefore, $v_{i,l}$ will be either adjacent the vertices corresponding
to part (a) and part (b) or (a) and part (c) or else  part (a) and part (d) according to whether
 $k_2<\frac{j}{2}$, $\frac{j}{2} \le k_2 <  j-1$ or $k_2 = j-1$ respectively. As
done before, in all these scenarios we get $d = 2k_1(j- 1) + 2k_2$ as required.

\vspace{6pt}

\noindent If $d=2k_1(j-1)+2k_2+1$ for some non negative integers $k_1$ and $k_2$ where such that $2k_1 \le s-2$ and $ k_2 < j-1$. Construct  a graph by connecting the vertices $v_{i,l}$ and $v_{p,r}$ if one of the following situations hold

\vspace{4pt}
\noindent \textbf{a)}  If $r \in B_{k_1,s}(l)$ and $p \neq i$.

\vspace{4pt}
\noindent \textbf{b)} If $k_2 \le \frac{j-2}{2}$, $r=l$ and ($p \in B_{k_2,j}(i)$ or $p \in \sigma_{\frac{j}{2},j}(i)$).

\vspace{4pt}

\noindent \textbf{c)}  If $k_2 \ge \frac{j}{2}$ and  (($r=l$ and  $p\neq i$) or  ($r \in \sigma_{\frac{s}{2},s}(l)$ and $p \in B_{\frac{2k_2-(j-2)}{2},j}(i)$)).

\vspace{2pt}

\begin{tikzpicture}[line cap=round,line join=round,>=triangle 45,x=1.0cm,y=1.0cm]
\clip(-4.994071446863002,0.5488293761738334) rectangle (5.072603329065303,6.851146917787157);
\draw [rotate around={87.91580736590971:(-3.3954971190781054,3.661631121281466)}] (-3.3954971190781054,3.661631121281466) ellipse (1.547757610453977cm and 0.8582776916516105cm);
\draw [rotate around={-89.17565717915494:(3.0162451984635084,3.5820530892448508)}] (3.0162451984635084,3.5820530892448508) ellipse (1.5866417120357255cm and 0.7648084220062332cm);
\draw (-3.972074007682463,2.933490067595092) node[anchor=north west] {$v_{0,0}$};
\draw [rotate around={0.0:(0.20363719590267826,5.790557437070943)}] (0.20363719590267826,5.790557437070943) ellipse (2.3718147816834696cm and 0.5206116257827407cm);
\draw [rotate around={-0.15358490033993571:(0.33820019206145846,1.5915020594965694)}] (0.33820019206145846,1.5915020594965694) ellipse (2.437053978284441cm and 0.4819305243903091cm);
\draw (-1.4000471190781076,6.868180208440167) node[anchor=north west] {$v_{1,0}$};
\draw (3.7610399487836124,4.960451655303161) node[anchor=north west] {$v_{2,0}$};
\draw (-3.989107298335472,4.347253191794838) node[anchor=north west] {$v_{0,2}$};
\draw (0.7461475032010234,6.851146917787158) node[anchor=north west] {$v_{1,2}$};
\draw (3.863239692701666,3.4444887871853616) node[anchor=north west] {$v_{2,2}$};
\draw (-3.972074007682463,3.6999881469804965) node[anchor=north west] {$v_{0,1}$};
\draw (0.4395482714468618,1.1109279677231307) node[anchor=north west] {$v_{3,1}$};
\draw (3.880272983354675,4.159886994611739) node[anchor=north west] {$v_{2,1}$};
\draw (-3.886907554417418,5.079684689874224) node[anchor=north west] {$v_{0,3}$};
\draw (1.6829784891165172,6.851146917787158) node[anchor=north west] {$v_{1,3}$};
\draw (3.846206402048657,2.780190451718011) node[anchor=north west] {$v_{2,3}$};
\draw (1.717045070422535,1.1279612583761398) node[anchor=north west] {$v_{3,0}$};
\draw (-0.3439830985915511,6.868180208440167) node[anchor=north west] {$v_{1,1}$};
\draw (-0.6165157490396947,1.1449945490291489) node[anchor=north west] {$v_{3,2}$};
\draw (-1.6555464788732424,1.1790611303351668) node[anchor=north west] {$v_{3,3}$};
\draw (1.7851782330345656,5.866237070938216)-- (-3.0693096030729885,4.760150114416477);
\draw (1.7851782330345656,5.866237070938216)-- (3.1989413572343097,2.513047139588101);
\draw (3.1989413572343097,2.513047139588101)-- (-1.4852135723431552,1.5001043478260867);
\draw (-1.4852135723431552,1.5001043478260867)-- (-3.0693096030729885,4.760150114416477);
\draw (-3.0522763124199797,4.038282837528605)-- (0.9675802816901353,5.877880091533182);
\draw (0.9675802816901353,5.877880091533182)-- (3.2330079385403274,3.1883423340961103);
\draw (3.2330079385403274,3.1883423340961103)-- (-0.48024942381562646,1.4418892448512581);
\draw (-0.48024942381562646,1.4418892448512581)-- (-3.0522763124199797,4.038282837528605);
\draw (3.2330079385403274,3.1883423340961103)-- (-3.0522763124199797,4.038282837528605);
\draw (-0.48024942381562646,1.4418892448512581)-- (0.9675802816901353,5.877880091533182);
\draw (1.7851782330345656,5.866237070938216)-- (-1.4852135723431552,1.5001043478260867);
\draw (-3.0693096030729885,4.760150114416477)-- (3.1989413572343097,2.513047139588101);
\draw (-3.103376184379006,3.3862736842105265)-- (-0.15661690140845608,5.9244521739130445);
\draw (-0.15661690140845608,5.9244521739130445)-- (3.2330079385403274,3.86363752860412);
\draw (3.2330079385403274,3.86363752860412)-- (0.6269144686299561,1.4418892448512584);
\draw (0.6269144686299561,1.4418892448512584)-- (-3.103376184379006,3.3862736842105265);
\draw (-3.103376184379006,3.3862736842105265)-- (3.2330079385403274,3.86363752860412);
\draw (-0.15661690140845608,5.9244521739130445)-- (0.6269144686299561,1.4418892448512584);
\draw (-3.154476056338033,2.617834324942792)-- (-1.1615810499359849,5.912809153318079);
\draw (-1.1615810499359849,5.912809153318079)-- (3.2159746478873186,4.608790846681923);
\draw (3.2159746478873186,4.608790846681923)-- (1.849904737516006,1.3999743707093837);
\draw (1.849904737516006,1.3999743707093837)-- (-3.154476056338033,2.617834324942792);
\draw (-3.154476056338033,2.617834324942792)-- (3.2159746478873186,4.608790846681923);
\draw (-1.1615810499359849,5.912809153318079)-- (1.849904737516006,1.3999743707093837);
\draw (-3.0693096030729885,4.760150114416477)-- (-0.15661690140845608,5.9244521739130445);
\draw (-3.0693096030729885,4.760150114416477)-- (0.6269144686299561,1.4418892448512584);
\draw (-3.0522763124199797,4.038282837528605)-- (-1.1615810499359849,5.912809153318079);
\draw (-3.0522763124199797,4.038282837528605)-- (1.849904737516006,1.3999743707093837);
\draw (-3.103376184379006,3.3862736842105265)-- (1.7851782330345656,5.866237070938216);
\draw (-3.103376184379006,3.3862736842105265)-- (-1.4852135723431552,1.5001043478260867);
\draw (-3.154476056338033,2.617834324942792)-- (0.9675802816901353,5.877880091533182);
\draw (-3.154476056338033,2.617834324942792)-- (-0.48024942381562646,1.4418892448512581);
\draw (-1.1615810499359849,5.912809153318079)-- (3.2330079385403274,3.1883423340961103);
\draw (-0.15661690140845608,5.9244521739130445)-- (3.1989413572343097,2.513047139588101);
\draw (0.9675802816901353,5.877880091533182)-- (3.2159746478873186,4.608790846681923);
\draw (1.7851782330345656,5.866237070938216)-- (3.2330079385403274,3.86363752860412);
\draw (3.2330079385403274,3.86363752860412)-- (-1.4852135723431552,1.5001043478260867);
\draw (3.2159746478873186,4.608790846681923)-- (-0.48024942381562646,1.4418892448512581);
\draw (3.2330079385403274,3.1883423340961103)-- (1.849904737516006,1.3999743707093837);
\draw (3.1989413572343097,2.513047139588101)-- (0.6269144686299561,1.4418892448512584);
\begin{scriptsize}
\draw [fill=black] (-3.0522763124199797,4.038282837528605) circle (1.5pt);
\draw [fill=black] (-3.103376184379006,3.3862736842105265) circle (1.5pt);
\draw [fill=black] (-3.154476056338033,2.617834324942792) circle (1.5pt);
\draw [fill=black] (3.2159746478873186,4.608790846681923) circle (1.5pt);
\draw [fill=black] (3.2330079385403274,3.86363752860412) circle (1.5pt);
\draw [fill=black] (3.2330079385403274,3.1883423340961103) circle (1.5pt);
\draw [fill=black] (-1.1615810499359849,5.912809153318079) circle (1.5pt);
\draw [fill=black] (-0.15661690140845608,5.9244521739130445) circle (1.5pt);
\draw [fill=black] (0.9675802816901353,5.877880091533182) circle (1.5pt);
\draw [fill=black] (-0.48024942381562646,1.4418892448512581) circle (1.5pt);
\draw [fill=black] (0.6269144686299561,1.4418892448512584) circle (1.5pt);
\draw [fill=black] (1.849904737516006,1.3999743707093837) circle (1.5pt);
\draw [fill=black] (-3.0693096030729885,4.760150114416477) circle (1.5pt);
\draw [fill=black] (1.7851782330345656,5.866237070938216) circle (1.5pt);
\draw [fill=black] (3.1989413572343097,2.513047139588101) circle (1.5pt);
\draw [fill=black] (-1.4852135723431552,1.5001043478260867) circle (1.5pt);
\end{scriptsize}
\end{tikzpicture}
\begin{center}
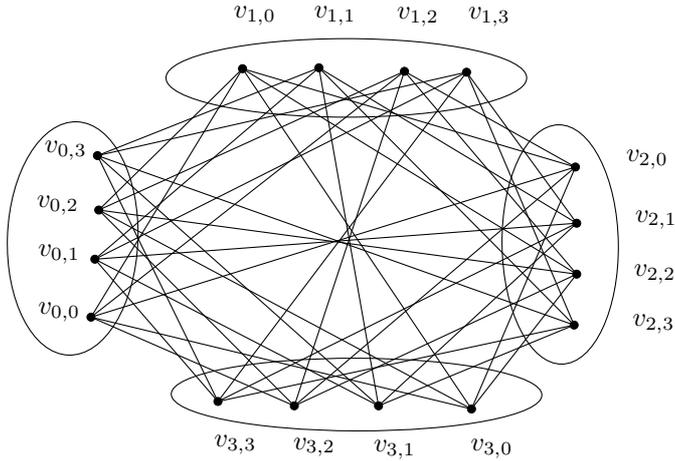
\noindent  Figure 9: In the case(2) when $d=5=2k_1(j-1)+2k_2+1=2\times 0\times (4-1)+2\times 2 +1$
\end{center}

\vspace{8pt}
\noindent It should be noted that the vertex sets of part (b) and part (c)
are disjoint. Therefore, $v_{i,l}$ will be either adjacent the vertices corresponding
to part (a) and part (b) or (a) and part (c) according to whether
$k_2 \le \frac{j-2}{2}$ or $ k_2 \ge  \frac{j}{2}$ respectively. As
done before, in all these scenarios we get $d = 2k_1(j- 1) + 2k_2+1$ as required.

\vspace{12pt}

\noindent \textbf{Case 3)} If $j$ is odd and $s$ is even.
\vspace{8pt}

\noindent  If $d=2k_1(j-1)+2k_2$ for some non negative integers $k_1$ and $k_2$ such that $2k_1 \le s-2$ and $0<k_2 \le j-1$. Construct  a graph by connecting the vertices $v_{i,l}$ and $v_{p,r}$ if one of the following situations hold

\vspace{4pt}
\noindent \textbf{a)} If $r \in B_{k_1,s}(l)$ and $p \neq i$.

\vspace{4pt}

\noindent \textbf{b)}  If $k_2 < \frac{j}{2}$, $r=l$ and $p \in B_{k_2,j}(i)$

\vspace{4pt}

\noindent \textbf{c)} If $k_2 \ge \frac{j}{2}$ and  (($r=l$ and $p\neq i$) or ($r \in \sigma_{\frac{s}{2},s}(l)$ and $p \in B_{\frac{2k_2-(j-1)}{2},j}(i)$)). 

\vspace{8pt}

\noindent It should be noted that the vertex sets of part (b) and part (c)
are disjoint. Therefore,  $v_{i,l}$ will be either adjacent the vertices corresponding
to part (a) and part (b) or (a) and part (c) according to whether
$k_2<\frac{j}{2}$ or $ k_2 \ge  \frac{j}{2}$ respectively. As
done before, in all these scenarios we get $d = 2k_1(j- 1) + 2k_2$ as required.

\vspace{12pt}

\noindent \textbf{Case 4)} If $j$ is odd and $s$ is even.
\vspace{8pt}

\noindent If $d=2k_1(j-1)+2k_2+1$ for some non negative integers $k_1$ and $k_2$ such that $2k_1 \le s-2$ and $0 \le k_2 < j-1$. Construct  a graph by connecting the vertices $v_{i,l}$ and $v_{p,r}$ if one of the following situations hold

\vspace{8pt}
\noindent \textbf{a)}  If $r \in B_{k_1,s}(l)$ and $p \neq i$.

\vspace{8pt}

\noindent \textbf{b)}  If $k_2 < \frac{j-1}{2}$ and  (($r=l$ and $p \in B_{k_2,j}(i)$) or ($r \in \sigma_{\frac{s}{2},s}(l)$ and $r>l$ and $p=\sigma^{+}_{k_2+1,j}(i)$) or ($r \in \sigma_{\frac{s}{2},s}(l)$ and $r<l$ and $p=\sigma^{-}_{k_2+1,j}(i)$)).

\vspace{8pt}

\noindent \textbf{c)}  If $k_2 = \frac{j-1}{2}$and  (($r=l$ and $p\neq i$) or ($r \in \sigma_{\frac{s}{2},s}(l)$ and $r>l$ and $p=\sigma^{+}_{k_2+1,j}(i)$) or ($r \in \sigma_{\frac{s}{2},s}(l)$ and $r<l$ and $p=\sigma^{-}_{k_2+1,j}(i)$)).

\vspace{8pt}

\noindent \textbf{d)}  If $k_2 > \frac{j-1}{2}$ and (($r=l$ and $p\neq i$) or ($r \in \sigma_{\frac{s}{2},s}(l)$  and  $p \in B_{\frac{2k_2-(j-1)}{2},j}(i)$) or  ($r \in \sigma_{\frac{s}{2},s}(l)$ and $p=\sigma^{+}_{\frac{2k_2-(j-1)}{2}+1,j}(i)$ with $r>l$ ) or  ($r \in \sigma_{\frac{s}{2},s}(l)$  and  $p=\sigma^{-}_{\frac{2k_2-(j-1)}{2}+1,j}(i)$  with $r<l$)).

\vspace{8pt}

\begin{tikzpicture}[line cap=round,line join=round,>=triangle 45,x=1.0cm,y=1.0cm]
\clip(-4.547799231754154,0.39052448512585997) rectangle (6.200207170294508,5.711384897025177);
\draw (0.017122663252250644,0.7863871853546925) node[anchor=north west] {$v_{2,2}$};
\draw [rotate around={31.31940798555165:(-2.3930879641485308,4.186149199084669)}] (-2.3930879641485308,4.186149199084669) ellipse (1.8193170789947273cm and 0.5790787938258221cm);
\draw [rotate around={-39.71631220657896:(4.036979257362354,4.110469565217392)}] (4.036979257362354,4.110469565217392) ellipse (1.7494760444172601cm and 0.7202335067485587cm);
\draw (5.280409475032023,4.209435240274604) node[anchor=north west] {$v_{1,2}$};
\draw (-3.321402304737508,5.094304805491995) node[anchor=north west] {$v_{0,2}$};
\draw (1.1583531370038518,0.7863871853546925) node[anchor=north west] {$v_{2,1}$};
\draw (4.752377464788745,4.814872311212819) node[anchor=north west] {$v_{1,1}$};
\draw (-4.053833802816894,4.616940961098402) node[anchor=north west] {$v_{0,1}$};
\draw (2.3847500640204977,0.7980302059496582) node[anchor=north west] {$v_{2,0}$};
\draw (4.088079129321395,5.385380320366138) node[anchor=north west] {$v_{1,0}$};
\draw (-4.4966993597951275,4.093005034324946) node[anchor=north west] {$v_{ 0,0}$};
\draw [rotate around={-0.24175315415687906:(0.6984542893726006,1.3801812356979415)}] (0.6984542893726006,1.3801812356979415) ellipse (2.828566326316242cm and 0.621612178572758cm);
\draw (0.3066886043533933,1.3336091533180794)-- (-2.657103969270166,4.116291075514878);
\draw (0.3066886043533933,1.3336091533180794)-- (3.8155464788732396,4.407366590389021);
\draw (-2.657103969270166,4.116291075514878)-- (3.1853147247119082,4.954588558352407);
\draw (-2.657103969270166,4.116291075514878)-- (4.2924786171574905,3.813572540045771);
\draw (-2.657103969270166,4.116291075514878)-- (2.333650192061459,1.2753940503432515);
\draw (1.396819206145967,1.25210800915332)-- (-3.236235851472471,3.673856292906182);
\draw (1.396819206145967,1.25210800915332)-- (-1.9587390524967987,4.616940961098403);
\draw (1.396819206145967,1.25210800915332)-- (3.1853147247119082,4.954588558352407);
\draw (1.396819206145967,1.25210800915332)-- (4.2924786171574905,3.813572540045771);
\draw (3.8155464788732396,4.407366590389021)-- (-1.9587390524967987,4.616940961098403);
\draw (3.8155464788732396,4.407366590389021)-- (-3.236235851472471,3.673856292906182);
\draw (3.8155464788732396,4.407366590389021)-- (2.333650192061459,1.2753940503432515);
\draw (-1.2092742637643954,0.8096732265446239) node[anchor=north west] {$v_{2,3}$};
\draw (-2.4016046094750236,5.571668649885589) node[anchor=north west] {$v_{0,3}$};
\draw (5.5529421254801665,3.6272842105263194) node[anchor=north west] {$v_{1,3}$};
\draw (-3.236235851472471,3.673856292906182)-- (4.769410755441742,3.371137757437075);
\draw (-1.9587390524967987,4.616940961098403)-- (4.769410755441742,3.371137757437075);
\draw (-1.328507298335467,4.942945537757442)-- (3.1853147247119082,4.954588558352407);
\draw (-1.328507298335467,4.942945537757442)-- (4.2924786171574905,3.813572540045771);
\draw (-1.328507298335467,4.942945537757442)-- (0.3066886043533933,1.3336091533180794);
\draw (-1.328507298335467,4.942945537757442)-- (2.333650192061459,1.2753940503432515);
\draw (-3.236235851472471,3.673856292906182)-- (-0.9878414852752878,1.3452521739130456);
\draw (-0.9878414852752878,1.3452521739130456)-- (-1.9587390524967987,4.616940961098403);
\draw (-0.9878414852752878,1.3452521739130456)-- (3.1853147247119082,4.954588558352407);
\draw (-0.9878414852752878,1.3452521739130456)-- (4.2924786171574905,3.813572540045771);
\draw (4.769410755441742,3.371137757437075)-- (0.3066886043533933,1.3336091533180794);
\draw (2.333650192061459,1.2753940503432515)-- (4.769410755441742,3.371137757437075);
\draw (-3.236235851472471,3.673856292906182)-- (4.2924786171574905,3.813572540045771);
\draw (2.333650192061459,1.2753940503432515)-- (-1.9587390524967987,4.616940961098403);
\draw (-2.657103969270166,4.116291075514878)-- (4.769410755441742,3.371137757437075);
\draw (3.1853147247119082,4.954588558352407)-- (0.3066886043533933,1.3336091533180794);
\draw (3.8155464788732396,4.407366590389021)-- (-0.9878414852752878,1.3452521739130456);
\draw (1.396819206145967,1.25210800915332)-- (-1.328507298335467,4.942945537757442);
\begin{scriptsize}
\draw [fill=black] (-1.9587390524967987,4.616940961098403) circle (1.5pt);
\draw [fill=black] (-3.236235851472471,3.673856292906182) circle (1.5pt);
\draw [fill=black] (0.3066886043533933,1.3336091533180794) circle (1.5pt);
\draw [fill=black] (4.2924786171574905,3.813572540045771) circle (1.5pt);
\draw [fill=black] (3.1853147247119082,4.954588558352407) circle (1.5pt);
\draw [fill=black] (3.8155464788732396,4.407366590389021) circle (1.5pt);
\draw [fill=black] (1.396819206145967,1.25210800915332) circle (1.5pt);
\draw [fill=black] (-2.657103969270166,4.116291075514878) circle (1.5pt);
\draw [fill=black] (2.333650192061459,1.2753940503432515) circle (1.5pt);
\draw [fill=black] (4.769410755441742,3.371137757437075) circle (1.5pt);
\draw [fill=black] (-1.328507298335467,4.942945537757442) circle (1.5pt);
\draw [fill=black] (-0.9878414852752878,1.3452521739130456) circle (1.5pt);
\end{scriptsize}
\end{tikzpicture}
\begin{center}
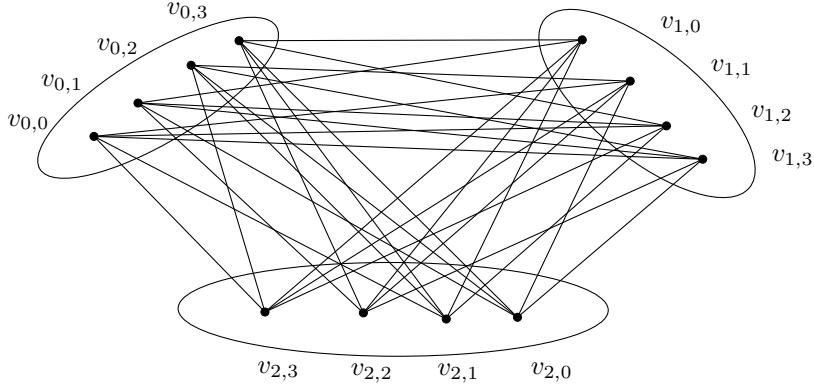
\noindent  Figure 10: In the case(3) when  $d=5=2k_1(j-1)+2k_2+1=2\times 1\times (3-1)+2\times 0 +1$
\end{center}

\vspace{8pt}

\noindent It should be noted that  the vertex sets of part (b), (c) and part (d)
are disjoint. Therefore,  $v_{i,l}$ will be either adjacent the vertices corresponding
to part (a) and part (b) or (a) and part (c) or else (a) and part (d) according to whether
$k_2<\frac{j-1}{2}$, $k_2 =  \frac{j-1}{2}$ or $ k_2 >  \frac{j-1}{2}$ respectively. As
done before, in all these scenarios we get $d = 2k_1(j- 1) + 2k_2 + 1$ as required.
\end{proof}

\vspace{5pt}

\begin{lemma}
\label{l1}
\noindent $m_j(S_n, S_{m}) \le  \left\lceil \dfrac{n+m-3}{j-1} \right\rceil$ where $j,n,m\ge3$.
\end{lemma}

\begin{proof}
Consider any red/blue coloring given by $K_{j \times s} =H_R \oplus H_B$, where  $s= \left\lceil \dfrac{n+m-3}{j-1}\right\rceil$, such that $H_R$ contains no red $S_n$.  Let $v$ be any vertex of $K_{j \times s}$. Then $v$ is  incident to at most $n-2$ red edge. Hence, \[d_B(v) \ge \left\lceil \dfrac{n+m-3}{j-1} \right\rceil (j-1) -(n-2) \ge m-1\] Therefore, $H_B$ will contain a blue $S_m$. Hence the result.
\end{proof} 

\vspace{5pt}

\begin{lemma}
\label{l2}
\noindent $m_j(S_n, S_{m}) \ge  \left\lceil \dfrac{n+m-4}{j-1} \right\rceil$ where $j,n,m\ge3$.
\end{lemma}

\begin{proof}
Consider the red and blue coloring of $K_{j \times s}$ given by $K_{j \times s} =H_R \oplus H_B$, where  $s= \left\lceil \dfrac{n+m-4}{j-1}\right\rceil-1$, where all the vertices will have uniform red degree of $n-2$ or $n-3$ (this is possible by lemma \ref{lemma 2}). Then clearly $H_B$ does not contain a red $S_n$.  Let $v$ be any vertex of $K_{j \times s}$. Then, 

\vspace{6pt}

$d_B(v) = \left(\left\lceil \dfrac{n+m-4}{j-1}\right\rceil -1 \right) (j-1) -(n-3)$ 

     $ $ $ $ $ $ $ $ $ $ $ $ $ $ $ $ $ =\left\lceil \dfrac{n+m-4}{j-1}\right\rceil (j-1)-j+1-n+3$ 

     $ $ $ $ $ $ $ $ $ $ $ $ $ $ $ $ $\ge n+m-4-j-n+4\ge m-j$

\vspace{6pt}

\noindent Therefore, $H_B$ will not contain a blue $S_m$. Hence the result.
\end{proof} 

\vspace{6pt}

\begin{lemma}
\label{l3}
\noindent $m_j(S_n, S_{m}) =  \left\lceil \dfrac{n+m-4}{j-1} \right\rceil$ if $ (n+m-4) \neq 0$ mod $(j-1)$ where $j,n,m\ge3$.
\end{lemma}

\begin{proof}
We know that if $(n+m-4)\neq 0$ mod $(j-1)$ then $\left\lceil \dfrac{n+m-4}{j-1} \right\rceil=\left\lceil \dfrac{n+m-3}{j-1} \right\rceil$. Hence the result follows by lemma \ref{l1} and lemma \ref{l2}.

\end{proof}

\vspace{6pt}

\begin{lemma}
\label{l4}
\noindent Suppose that $j,n,m\ge3$. Then, $m_j(S_n, S_{m}) \le  \left\lceil \dfrac{n+m-4}{j-1} \right\rceil$ provided that $ (n+m-4) = 0$ mod $(j-1)$ with $j$ is odd, $n$ is odd and $s=\dfrac{n+m-4}{j-1}$ is odd.
\end{lemma}

\vspace{6pt}

\begin{proof}
Consider any red/blue coloring given by $K_{j \times s} =H_R \oplus H_B$, where  $s= \left\lceil \dfrac{n+m-4}{j-1}\right\rceil$, such that $H_R$ contains no red $S_n$.  Since, $j \times s \times (n-2)$ is odd, there will exist at least one vertex $v \in K_{j \times s}$ such it is not incident to $n-2$ red edges, as otherwise by handshake lemma $j \times s \times (n-2)=2|E(H_R)|$, a contradiction. Hence, \[d_B(v) \ge \left\lceil \dfrac{n+m-4}{j-1} \right\rceil (j-1) -(n-3) \ge m-1\] Therefore, $H_B$ will contain a blue $S_m$. Hence the result.
\end{proof}  

\vspace{6pt}

\begin{lemma}
\label{l5}
\noindent Suppose that $j,n,m\ge3$. Then, $m_j(S_n, S_{m}) \ge  \left\lceil \dfrac{n+m-3}{j-1} \right\rceil$ provided that $ (n+m-4) = 0$ mod $(j-1)$ with $j$ is even or $s=\dfrac{n+m-4}{j-1}$ even or $n$ is even.

\end{lemma}

\begin{proof}
By lemma \ref{l1} and lemma \ref{l2}, $K_{j \times s}$ where $s= \left\lceil \dfrac{n+m-3}{j-1}\right\rceil-1=\left\lceil \dfrac{n+m-4}{j-1}\right\rceil$, will have a $n-2$ regular subgraph on $K_{j \times s}$. Using this subgraph generate a red/blue coloring given by $K_{j \times s} =H_R \oplus H_B$,  where all the edges of this subgraph are colored red and all other edges colored blue. Then clearly $H_R$ is $S_n$ - free. Furthermore, for any vertex $v \in K_{j \times s}$,  $d_B(v) =\left( \dfrac{n+m-4}{j-1} \right) (j-1)-(n-2)  = m-2$. Therefore, $H_B$ will not contain a blue $S_m$. Hence the result.
\end{proof} 
\vspace{6pt}

\begin{theorem} 
If $j\ge3$ and $n,m\ge 2$ then, \[
m_j(S_n, S_m) = 
\begin{cases} 
\left\lceil \dfrac{\max \{n,m\}-1}{j-1} \right\rceil & \text{ if }  n=2 \text{ or }  m=2  \\
\hspace{14pt} & \\
\left\lceil \dfrac{n+m-4}{j-1} \right\rceil &  \text{ if }  n+m-4=(j-1)s;  j,s,n  \\                                        &                                  \text{ are odd and }        n,m \ge 3\\
\hspace{14pt} & \\
\left\lceil \dfrac{n+m-3}{j-1} \right\rceil & \text{ otherwise } \\
\end{cases}
\]
\end{theorem}

\begin{proof}
The theorem clearly follows from lemmas \ref{l3}, \ref{l4} and \ref{l5} as $m_j(S_2,S_m)=\left\lceil \dfrac{m-1}{j-1} \right\rceil$ (see Syafrizal et al 2005).
\end{proof}

\end{document}